\documentclass{amsproc}

\usepackage{graphicx}

\usepackage{tkz-graph}
\usepackage{amsmath, amstext, amsgen, amsbsy, amsopn, amsfonts, amssymb,psfrag, amsthm, epstopdf}

\newtheorem{theorem}{Theorem}[section]
\newtheorem{lemma}[theorem]{Lemma}

\theoremstyle{definition}
\newtheorem{definition}[theorem]{Definition}

\theoremstyle{remark}

\newtheorem*{question}{Question}
\theoremstyle{lemma}
\newtheorem*{subgroupcor}{Subgroup Theorem}
\newtheorem*{smith}{Smith Theory}
\newtheorem*{finiteorder}{Finite Order Theorem}
\newtheorem*{isometry}{Isometry Theorem}
\newtheorem*{subgraph}{Subgraph Lemma}
\newtheorem*{pathlemma}{Path Lemma}
\newtheorem*{realizable}{Realizability Lemma}

\numberwithin{equation}{section}

\def\Z{\mathbb{Z}}

\def\a{{\alpha}}
\def\b{{\beta}}
\def\g{{\gamma}}
\def\d{{\delta}}
\newcommand{\semi}{{\rtimes}}

\def\TSG{{\mathrm{TSG}}}
\def\fix{{\mathrm{fix}}}
\def\Aut{{\mathrm{Aut}}}

\begin{document}

\title{Topological Symmetries of the Heawood family}


\author{Blake Mellor}
\address{Loyola Marymount University, 1 LMU Drive, Los Angeles, CA 90045}
\email{blake.mellor@lmu.edu}
\curraddr{}
\thanks{}

\author{Robin Wilson}
\address{Loyola Marymount University, 1 LMU Drive, Los Angeles, CA 90045}
\curraddr{}
\email{robin.wilson@lmu.edu}
\thanks{}

\subjclass[2020]{Primary 57M15}

\date{}

\begin{abstract} 
The {\em topological symmetry group} of an embedding $\Gamma$ of an abstract graph $\gamma$ in $S^3$ is the group of automorphisms of $\gamma$ which can be realized by homeomorphisms of the pair $(S^3, \Gamma)$.  These groups are motivated by questions about the symmetries of molecules in space.  In this paper, we find all the groups which can be realized as topological symmetry groups for each of the graphs in the Heawood family. This is an important collection of spatial graphs, containing the only intrinsically knotted graphs with 21 or fewer edges.  As a consequence, we discover that the graphs in this family are all intrinsically chiral.
\end{abstract}

\maketitle

\section{Introduction}\label{S:intro}

The chemical properties of certain molecules are often strongly influenced by their symmetries in space. Molecules with the same chemical structure but whose embeddings in space are not equivalent are called {\bf stereoisomers}, and can often have quite different properties. The subfield of {\em stereochemistry} focuses on the spatial structure of molecules, and particularly on the study of stereoisomers. Historically, chemists have been most interested in rigid symmetries of molecules, but as our ability to synthesize molecules advances there is increased interest in symmetries of more ``flexible" molecules \cite{ro}. To describe the symmetries of these more complex molecules, Jonathan Simon \cite{si} introduced the {\bf topological symmetry group} of an embedded graph as the group of automorphisms of the graph induced by homeomorphisms of $S^3$.

Since Simon's original paper, considerable work has been done on topological symmetry groups.  As a generalization of the study of symmetries and achirality of knots and links, this is a natural and important problem in low dimensional topology. In many cases, the motivating question is: given an abstract graph (e.g. chemical structure of a molecule), what are the possible topological symmetry groups over all embeddings of the graph in $S^3$? This problem has been addressed for several important families of graphs, including complete graphs \cite{cf, cfo, fmn2, fmn3, fnt},  complete bipartite graphs \cite{flmpv, hmp, me} and M\"{o}bius ladders \cite{f1, fl}, as well as for some individually interesting graphs \cite{cfhltv, efmsw, flw}.  There are also some broad results restricting possible topological symmetry groups for any 3-connected graph \cite{fnpt}.

The Heawood family of graphs, shown in Figure \ref{F:Heawood}, is an important family of graphs in low-dimensional topology, consisting of all graphs which can be derived from the complete graph $K_7$ by $\nabla Y$ or $Y\nabla$ moves. A $\nabla Y$ move is an operation in which a triangle in the graph is replaced by a degree-three vertex in the shape of a ``Y", while a $Y\nabla$ move is the reverse, as shown in Figure \ref{F:triangleY}. Aside from $K_7$, the best known graph in the family is the Heawood graph (denoted $C_{14}$). The 14 graphs of the family derived from $K_7$ only by $\nabla Y$ moves are all intrinsically knotted -- in fact, they are the only intrinsically knotted graphs with 21 or fewer edges \cite{lklo}.  The other 6 graphs, denoted $N_i$ and $N'_i$, are not intrinsically knotted; they were the first examples demonstrating that $Y\nabla$ moves do not preserve intrinsic knottedness \cite{fn}.

\begin{figure} [tbp]
$$\scalebox{.7}{\includegraphics{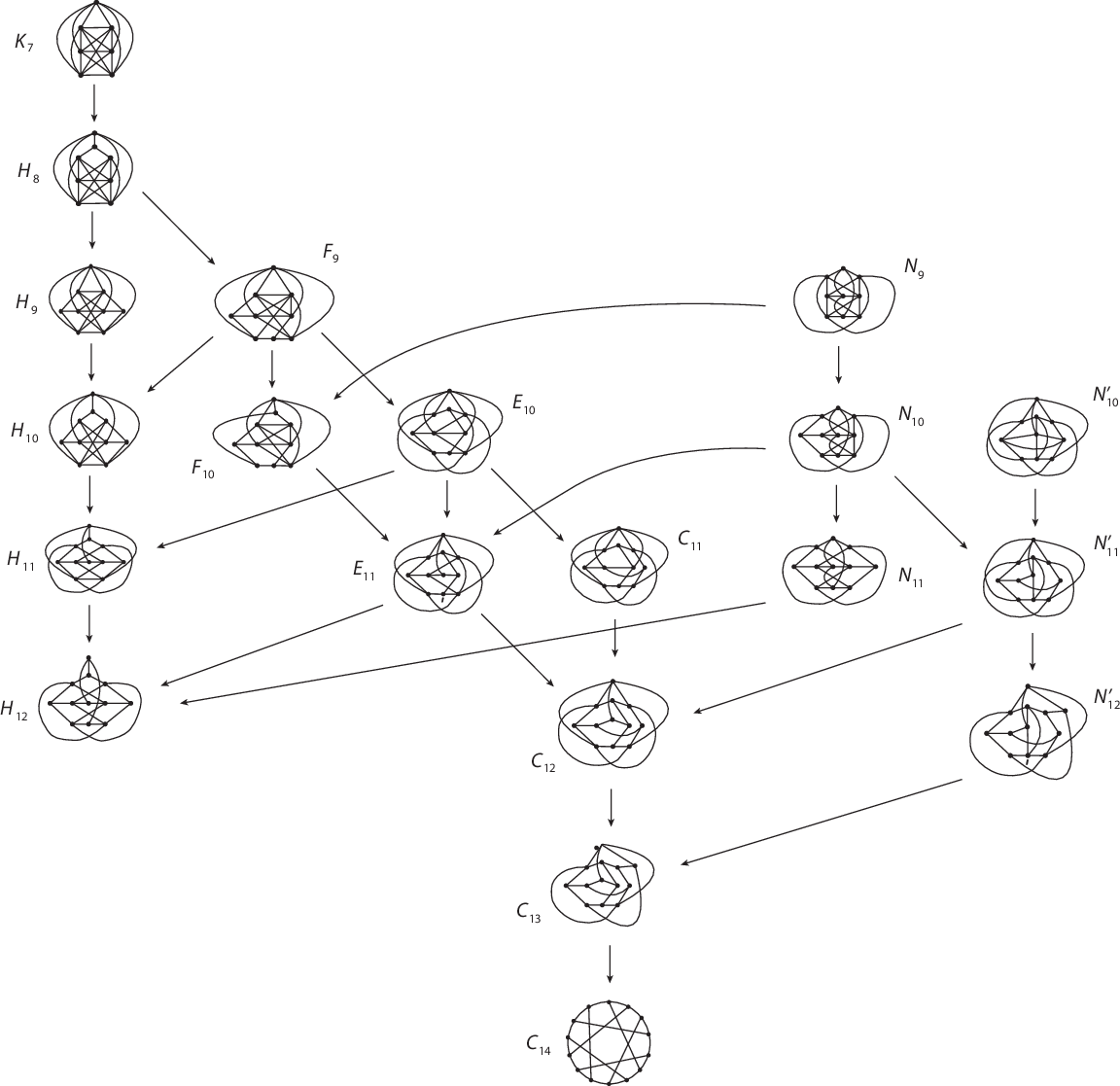}}$$
\caption{The Heawood family of graphs (adapted from \cite{lklo}).  Arrows indicate $\nabla Y$ moves between graphs.}
\label{F:Heawood}
\end{figure}

\begin{figure} [htbp]
$$\scalebox{.7}{\includegraphics{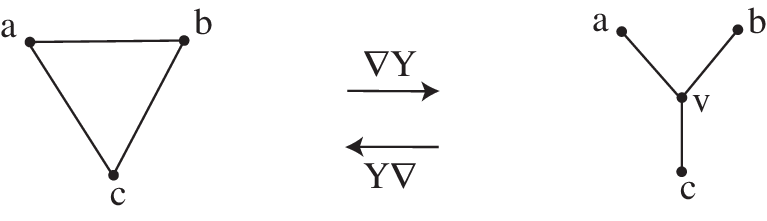}}$$
\caption{$\nabla Y$ and $Y\nabla$ moves.}
\label{F:triangleY}
\end{figure}

Our main purpose in this paper is to determine what groups can occur as topological symmetry groups for some embedding of each graph in the Heawood family of graphs.  Aside from increasing the library of spatial graphs whose topological symmetry groups are known, the Heawood family provides an opportunity to explore how the $\nabla Y$ and $Y\nabla$ moves affect the symmetry groups. The classification of topological symmetry groups has been done for $K_7$ \cite{fmn3} and the Heawood graph $C_{14}$ \cite{flw}.  Baza \cite{ba} classified the topological symmetry group for the graphs we denote by $H_8$ and $F_9$, though we will include our own arguments here.  Our names for the graphs are those used by Lee, Kim, Lee and Oh \cite{lklo}.  Table \ref{Ta:TSG} summarizes our results, giving for each graph the automorphism group and the list of positively realizable topological symmetry groups (i.e., groups of automorphisms that are induced by orientation-preserving homeomorphisms of a single embedding of the graph in $S^3$). One of our most surprising observations is that, for the graphs in the Heawood family, every realizable group is also positively realizable.  In fact, our methods allow us to show that all the graphs in the family are intrinsically chiral (i.e., for each graph in the family, there is no embedding of the graph that is ambient isotopic to its mirror image). In Section \ref{S:background}, we give the necessary definitions, and the results we will be using from previous work.  Section \ref{S:details} contains the details of the arguments for each graph, along with embeddings realizing the various groups. Finally, in Section \ref{S:conclusion} we connect our results to intrinsic chirality, and pose some questions for future work.

\begin{table}[htbp]
\begin{tabular}{|c|c|c|}
\hline 
Graph & Automorphism group & Positively realizable  \\ \hline
$K_7$ & $S_7$ & $D_7, D_5, D_3, \Z_7, \Z_5, \Z_3, \Z_2$  \\ \hline
$H_8$ & $S_3 \times S_4$ & $\Z_3, \Z_2$  \\ \hline
$H_9$ & $(D_3 \times D_3) \semi \Z_2$ & $D_3, \Z_3, \Z_2$  \\ \hline
$H_{10}$ & $D_4$ & $\Z_2$  \\ \hline
$H_{11}$ & $D_2$ & $\Z_2$  \\ \hline
$H_{12}$ & $D_3 \times \Z_2$ & $D_3, \Z_3, \Z_2$  \\ \hline
$F_9$ & $D_4 \times \Z_2$ & $\Z_2$ \\ \hline
$F_{10}$ & $(\Z_2)^3 \semi D_3$ & $D_3, \Z_3, \Z_2$ \\ \hline
$E_{10}$ & $D_3$ & $\Z_3$  \\ \hline
$E_{11}$ & $D_3$ & $\Z_3$ \\ \hline
$C_{11}$ & $S_4$ & $\Z_3$  \\ \hline
$C_{12}$ & $D_4$ & trivial  \\ \hline
$C_{13}$ & $S_4$ & $\Z_3$  \\ \hline
$C_{14}$ & ${\rm PGL}(2,7)$ & $D_7, D_3, \Z_7, \Z_6, \Z_3, \Z_2$  \\ \hline
$N_9$ & $(\Z_2)^3 \semi D_3$ & $D_3 \times \Z_2, \Z_6, D_3, \Z_3, D_2, \Z_2$  \\ \hline
$N_{10}$ & $D_3$ & $D_3, \Z_3, \Z_2$  \\ \hline
$N_{11}$ & $D_3 \times \Z_2$ & $D_3 \times \Z_2, \Z_6, D_3, \Z_3, D_2, \Z_2$  \\ \hline
$N'_{10}$ & $D_8$ & $D_2, \Z_2$  \\ \hline
$N'_{11}$ & $D_2$ & $\Z_2$   \\ \hline
$N'_{12}$ & $D_6$ & $D_6, D_3, D_2, \Z_6, \Z_3, \Z_2$  \\ \hline
\end{tabular}
\medskip
\label{Ta:TSG}
\caption{Realizable topological symmetry groups for graphs in the Heawood family. All realizable groups for these graphs are positively realizable.}
\end{table}

\section{Background and definitions} \label{S:background}

Before we begin, we introduce some terminology, and some important tools from other papers.  An {\em abstract graph} (or just a {\em graph}) $\gamma$ is a pair $(V, E)$, where $V$ is a set of vertices and $E \subseteq V \times V$ is a set of edges. We will denote the edge between vertices $v$ and $w$ by $\overline{vw}$ (since our graphs are not directed, $\overline{vw} = \overline{wv}$). We should note that graph theorists usually represent an edge just as an unordered pair $\{v, w\}$; however, since we are looking at graph embeddings, we want to distinguish between a pair of vertices $\{v, w\}$ that will be embedded as two points, and an edge $\overline{vw}$ that will be embedded as an arc. An embedding $f: \gamma \rightarrow S^3$ means (1) an embedding of the vertices $V$ in $S^3$ and (2) for each edge $\overline{vw}$, an embedding $f_{v,w}: [0,1] \rightarrow S^3$ such that $f_{v,w}(0) = f(v)$ and $f_{v,w}(1) = f(w)$, and the embeddings of distinct edges intersect only at the endpoints. The image $\Gamma = f(\gamma)$ is called an {\em embedded graph} or a {\em spatial graph}.

An {\em automorphism} of a graph $\gamma$ is a bijection $\alpha: V \rightarrow V$ such that $\overline{vw} \in E$ if and only if $\overline{\alpha(v)\alpha(w)} \in E$. The automorphisms of a graph $\gamma$ form a group (the {\em automorphism group}), denoted $\Aut(\gamma)$. To describe automorphism groups (and their subgroups), we will use the following standard terms:
\begin{itemize}
\item The {\em symmetric group} $S_n$ is the group of permutations of a set of $n$ objects.
\item The {\em dihedral group} $D_n$ is the group of size $2n$ with presentation $\langle m, r \mid m^2 = r^n = 1, rm = mr^{-1}\rangle$.
\item The {\em cyclic group} $\Z_n$ is the group of size $n$ with presentation $\langle r \mid r^n = 1\rangle$.
\item The {\em direct product} $G \times H$ is the Cartesian product $G \times H$ with the group operation $(g_1, h_1)(g_2, h_2) = (g_1g_2, h_1h_2)$.
\item The {\em semidirect product} $G \semi H$ is the Cartesian product $G\times H$ with the group operation $(g_1, h_1)(g_2, h_2) = (g_1\phi(h_1)(g_2), h_1h_2)$, for some homomorphism $\phi: H \rightarrow \Aut(G)$. 
\end{itemize}

A homeomorphism of $S^3$ which takes an embedded graph $\Gamma$ to itself (taking vertices to vertices and edges to edges) induces an automorphism on the underlying abstract graph. We are interested in which automorphisms can be induced in this way for some embedding of a given graph. In what follows, we will refer to a homeomorphism of $S^3$ taking an embedded graph $\Gamma$ to itself (taking vertices to vertices and edges to edges) as a homeomorphism of the pair $(S^3,\Gamma)$.

\begin{definition}\label{def1.2} 
Let $\gamma$ be an abstract graph, and let $\Gamma$ be an embedding of $\gamma$ in $S^3$.  We define the {\bf topological symmetry group} $\TSG(\Gamma)$ as the subgroup of $\Aut(\gamma)$ induced by homeomorphisms of $(S^3,\Gamma)$.  We define the {\bf orientation preserving topological symmetry group} $\TSG_+(\Gamma)$ as the subgroup of $\Aut(\gamma)$ induced by orientation preserving homeomorphisms of $(S^3,\Gamma)$.  
\end{definition}

\begin{definition}\label{def1.4} Let  $G$ be a group and let $\gamma$ denote an abstract graph. If there is some embedding $\Gamma$ of $\gamma$ in $S^3$ such that $\TSG(\Gamma)=G$ (resp. $\TSG_+(\Gamma)=G$),  then we say that the group $G$ is {\bf realizable}  (resp. {\bf positively realizable}) for $\gamma$. We will also say that a particular automorphism $\sigma \in \Aut(\gamma)$ is {\bf realizable} (resp. {\bf positively realizable} or {\bf negatively realizable}) if there is an embedding $\Gamma$ of $\gamma$ and a homeomorphism (resp. orientation-preserving homeomorphism or orientation-reversing homeomorphism) of $(S^3, \Gamma)$ which induces $\sigma$.
\end{definition}

Note that for any embedding $\Gamma$ of an abstract graph, either $\TSG_+(\Gamma) = \TSG(\Gamma)$ or $\TSG_+(\Gamma)$ is an index two subgroup of $\TSG(\Gamma)$.  Further, for graphs that contain at least one cycle (i.e. not trees), if a group $G$ is positively realizable by some embedding $\Gamma$, then $G$ is also realizable \cite{cf}. This is because we can add identical chiral knots to every edge of the embedding $\Gamma$ to rule out any orientation-reversing homeomorphisms.  This yields a new embedding $\Gamma'$ with $\TSG(\Gamma') = \TSG_+(\Gamma') = \TSG_+(\Gamma) = G$.  Similarly, we can add distinct knots to every edge of an embedding $\Gamma$ to get a new embedding $\Gamma'$ such that $\TSG(\Gamma') = \TSG_+(\Gamma')$ is the trivial group; so the trivial group is always both realizable and positively realizable.  As a result, for any graph $\gamma$, we are interested in finding (1) the nontrivial groups that are positively realizable for $\gamma$ and (2) the nontrivial groups that are realizable, but not positively realizable, for $\gamma$.

For example, these groups have been found for $K_7$ and the Heawood graph $C_{14}$.  Both $K_7$ and $C_{14}$ are {\bf intrinsically chiral} \cite{ffn}, meaning that no automorphism of the graph is negatively realizable. So every group that is realizable for these graphs is positively realizable.

\begin{theorem} \cite{fmn3}
The nontrivial groups that can be (positively) realized for the complete graph $K_7$ are:
$$D_7, D_5, D_3, \Z_7, \Z_5, \Z_3, \Z_2.$$
\end{theorem}

\begin{theorem} \cite{flw}
The nontrivial groups that can be (positively) realized for the Heawood graph $C_{14}$ are:
$$D_7, D_3, \Z_7, \Z_6, \Z_3, \Z_2.$$
\end{theorem}

We now turn to some of the important tools we will use in this paper. The following result is very useful in showing that certain groups can be realized as a topological symmetry group. In many cases, it tells us that if a group is realizable for a graph $\gamma$, then all of its subgroups are also realizable. (In \cite{fmn1}, this is called the Subgroup Corollary.) Recall that a graph is {\em 3-connected} if removing any two vertices leaves a connected graph.

\begin{subgroupcor}\cite{fmn1} 
\label{subgroup}
Let $\Gamma$ be an embedding of a 3-connected graph in $S^3$.  Suppose that $\Gamma$ contains an edge $e$ that is not pointwise fixed by any nontrivial element of $\TSG_+(\Gamma)$.  Then for every $H \leq \TSG_+(\Gamma)$, there is an embedding $\Gamma'$ of $\Gamma$ with $H = \TSG_+(\Gamma')$.
\end{subgroupcor}

The following result is important for restricting which automorphisms or groups can be realized for a particular graph $\gamma$.

\begin{finiteorder}\cite{f2}
Let $\phi$ be a nontrivial automorphism of a 3-connected graph $\gamma$ that is induced by a homeomorphism $h$ of $(S^3, \Gamma)$ for some embedding $\Gamma$ of $\gamma$. Then there exists another embedding $\Gamma'$ of $\gamma$ such that the automorphism $\phi$ is induced by a {\em finite order} homeomorphism $f$ of $(S^3, \Gamma')$, and $f$ is orientation reversing if and only if $h$ is orientation reversing. 
\end{finiteorder}

Due to the Finite Order Theorem, when considering individual automorphisms, we need only determine whether or not the automorphism is realizable by a finite order homeomorphism.  In particular, this allows us to use the following result of P.A. Smith, which is one of our most powerful tools to determine that certain homeomorphisms are not positively or negatively realizable. Here $\fix(h)$ denotes the fixed point set of a homeomorphism $h: S^3 \rightarrow S^3$.

\begin{smith} \cite{sm}
Let $h$ be a nontrivial finite order homeomorphism of $S^3$.  If $h$ is orientation preserving, then $\fix(h)$ is either the empty set or is homeomorphic to $S^1$.  If $h$ is orientation reversing, then $\fix(h)$ is homeomorphic to either $S^0$ (two distinct points) or $S^2$.
\end{smith}

In particular, we note that if $h$ is orientation reversing with $\fix(h) \cong S^2$, then $h$ must have order 2 (since $h^2$ will be orientation preserving with fixed point set larger than $S^1$, and hence trivial). We can combine the Finite Order Theorem and Smith Theory to give criteria for when an automorphism is positively or negatively realizable.

\begin{definition}
Let $\phi$ be an automorphism of a graph $\gamma$. Let $F = A \cup B$ where $A$ is the subgraph of $\gamma$ induced by the fixed vertices of $\phi$, and $B$ is the set of midpoints of edges whose endpoints are interchanged by $\phi$ (so $B$ is a set of discrete points). We will call $F$ the {\bf fixed subgraph under $\phi$}.
\end{definition}

\begin{realizable}
Let $\phi$ be a nontrivial automorphism of a 3-connected nonplanar graph $\gamma$, and let $F$ be the fixed subgraph under $\phi$. If $\phi$ is positively realizable, then $F$ embeds in $S^1$. If $\phi$ is negatively realizable, then $\phi$ has even order.  Moreover, if $\phi$ is negatively realizable and $F$ contains more than two points, then $\phi$ has order 2 and $F$ is planar.
\end{realizable}
\begin{proof}
Suppose $\phi$ is realizable by a homeomorphism $h$ of $(S^3, \Gamma)$ for some embedding $\Gamma$ of $\gamma$.  Then, by the Finite Order Theorem, there exists another embedding $\Gamma'$ of $\gamma$ such that the automorphism $\phi$ is induced by a finite order homeomorphism $f$ of $(S^3, \Gamma')$, and $f$ is orientation reversing if and only if $h$ is orientation reversing.

If $f$ is orientation preserving, then $\fix(f)$ is either the empty set or homeomorphic to $S^1$ by Smith Theory.  Since the fixed subgraph under $\phi$ is a subset of $\fix(f)$, it must embed in $S^1$.  

If $f$ is orientation reversing, then $\fix(f)$ is homeomorphic to either $S^0$ or $S^2$.  Since the fixed subgraph $F$ under $\phi$ is a subset of $\fix(f)$, if $F$ contains more than two points, then $\fix(f)$ is homeomorphic to $S^2$. This implies $F$ is planar.  In addition, if $\fix(f)$ is homeomorphic to $S^2$, then $f$ is a reflection with order 2.  The order of $\phi$ must divide the order of $f$, so then $\phi$ also has order 2.

Finally consider the case when $\fix(f)$ is homeomorphic to $S^0$ (so $F$ contains two or fewer points). Suppose $f$ has order $m$, and $\phi$ has order $n$; there is some integer $d$ such that $m = nd$.  So $f^n$ has order $d$.  But $f^n$ realizes $\phi^n = id$, so it fixes the embedded graph $\Gamma'$.  Since $\gamma$ is non-planar, this means the fixed set of $f^n$ does not embed in $S^2$; hence, by Smith Theory, $f^n$ is trivial.  So $d = 1$, and $m = n$.  Since $f$ is orientation reversing, it has even order, so we conclude that $\phi$ also has even order.
\end{proof}

Note that all the graphs in the Heawood family are 3-connected and nonplanar, so the Realizability Lemma applies. From the Realizability Lemma, we can immediately conclude: \begin{enumerate}
	\item If the fixed subgraph under $\phi \in \Aut(\gamma)$ does not embed in $S^1$, then $\phi$ is not positively realizable.
	\item If the fixed subgraph under $\phi \in \Aut(\gamma)$ is nonplanar, then $\phi$ is not positively or negatively realizable.
	\item If $\phi$ has odd order, then $\phi$ is not negatively realizable.
	\item If the fixed subgraph under $\phi \in \Aut(\gamma)$ contains more than two points, and $\phi$ does not have order 2, then $\phi$ is not negatively realizable.
\end{enumerate}

The following strengthening of the Finite Order Theorem allows us to assume that if a group of automorphisms is positively realizable, then they can all simultaneously be realized by finite order homeomorphisms of the same embedding of the graph in $S^3$.

\begin{isometry} \cite{fmn2}
Let $\Gamma$ be an embedded 3-connected graph, and let $H\leq \TSG_+(\Gamma)$. Then $\Gamma$ can be re-embedded as $\Gamma'$ such that $H \leq \TSG_+(\Gamma')$ and $\TSG_+(\Gamma')$ is induced by an isomorphic finite subgroup of ${\rm SO}(4)$ (i.e., the group of orientation preserving isometries of $S^3$).
\end{isometry}

The next lemma, the Subgraph Lemma, simply says if a group is realizable for a graph, it must be realizable for every stable subgraph.  This lemma was given in \cite{efmsw}, but the statement there was missing a necessary hypothesis (which was satisfied whenever the lemma was applied in that paper).  We provide a corrected statement and proof here.

\begin{subgraph}
Let $\gamma$ be an abstract graph and $\gamma'$ be a subgraph of $\gamma$ so that every automorphism $\a$ of $\gamma$ fixes $\gamma'$ setwise.  Further assume that the restriction map $\phi: \Aut(\gamma) \rightarrow \Aut(\gamma')$ defined by $\phi(\a) = \a\vert_{\gamma'}$ is injective (i.e., every automorphism of $\gamma$ is determined by its action on $\gamma'$). Assume that $\Gamma$ is an embedding of $\gamma$ and let $\Gamma'$ be the embedding of $\gamma'$ induced by $\Gamma$.  Then $\TSG(\Gamma) \leq \TSG(\Gamma')$ and $\TSG_+(\Gamma) \leq \TSG_+(\Gamma')$.  
\end{subgraph}
\begin{proof}
Every element $\a \in \TSG(\Gamma)$ is realized by a homeomorphism of $S^3$ taking $\Gamma$ to itself setwise.  Hence, every such homeomorphism also takes $\Gamma'$ to itself setwise and realizes $\phi(\alpha) \in \Aut(\gamma')$. So $\phi$ induces a map from $\TSG(\Gamma)$ to a subgroup of $\TSG(\Gamma')$. Since $\phi$ is injective, $\TSG(\Gamma)$ is isomorphic to a subgroup of $\TSG(\Gamma')$.  The same argument holds for $\TSG_+(\Gamma)$. 
\end{proof}

The Path Lemma is a new tool which provides an easy criterion for proving that an automorphism is not negatively realizable.

\begin{pathlemma} 
Suppose $\alpha$ is an automorphism of a graph $\gamma$ whose fixed subgraph contains more than two points, and $a$ and $b$ are vertices of $\gamma$ interchanged by $\alpha$. Further suppose $P = av_1v_2\dots v_kb$ is a path in $\gamma$ such that no vertex $v_i$ is fixed by $\alpha$, and no edge $\overline{v_iv_{i+1}}$ is flipped by $\alpha$. Then $\alpha$ is not negatively realizable.
\end{pathlemma}
\begin{proof}
Suppose $\alpha$ is negatively realizable by $h$; by the Finite Order Theorem we may assume $h$ has finite order.  Since $\alpha$ fixes more than two points, by Smith Theory $h$ must have fixed point set homeomorphic to $S^2$, and $h$ interchanges the inside and outside of the fixed $S^2$.  Without loss of generality, suppose $a$ is inside the sphere and $b$ is outside the sphere.  Then the embedding of the path $P$ from $a$ to $b$ must cross the sphere at some point.  If it crosses the sphere at a vertex $v_i$, then this vertex is fixed.  If it crosses the sphere at a point inside the edge $\overline{v_iv_{i+1}}$, then this edge must be mapped to itself by $h$ (since the graph is embedded, no two edges can share an interior point). If the endpoints of the edge are not fixed, then they must be interchanged. So either there are fixed vertices along the path, or a flipped edge.
\end{proof}

Our final observation allows us, in many case, to only consider automorphism up to conjugacy.

\begin{lemma}
If $\a$ and $\b \in \Aut(\gamma)$ are conjugate, then $\b$ is positively (resp. negatively) realizable if and only if $\a$ is positively (resp. negatively) realizable.
\end{lemma}
\begin{proof}
Suppose $\a$ is realizable, and $\b = \phi^{-1}\a\phi$ for some $\phi \in \Aut(\gamma)$.  Then there is an embedding $\Gamma = f(\gamma)$ (where $f: \gamma \rightarrow S^3$) and a homeomorphism $h$ of $(S^3, \Gamma)$ that realizes $\a$. Define a new embedding $\Gamma' = f'(\gamma)$ by (1) for each vertex $v$ of $\gamma$, let $f'(v) = f(\phi(v))$ and (2) for each edge $\overline{vw}$ of $\gamma$, let $f'_{v,w} = f_{\phi(v)\phi(w)}$.  Since $\phi$ is an automorphism, it is a bijection between vertices, and $\overline{vw}$ is an edge of $\gamma$ if and only if $\overline{\phi(v)\phi(w)}$ is an edge of $\gamma$, so $f'$ is well-defined.  In fact, $\Gamma'$ is the same subset of $S^3$ as $\Gamma$, except that the vertices have been relabeled, so the vertex labeled $v$ in $\Gamma'$ is labeled $\phi^{-1}(v)$ in $\Gamma$.

So, $h$ is also a homeomorphism of $(S^3, \Gamma')$.  For each vertex $v$ in $\gamma$, we know that $h(f'(v)) = h(f(\phi(v))) = f(\a\phi(v))$ (since $h$ realizes $\a$ on $\Gamma$).  But since $\a\phi(v) = \phi\b(v)$, this means $h(f'(v)) = f(\phi\b(v)) = f'(\b(v))$.  So $h$ realizes $\b$ on $\Gamma'$.  Therefore $\b$ is also realizable, and is positively or negatively realizable as $\a$ is (since we are using the same homeomorphism of $S^3$).
\end{proof}

\section{Computing the topological symmetry groups for the Heawood graphs}\label{S:details}

\SetGraphUnit{.5}
\GraphInit[vstyle=Simple]
\SetVertexSimple[MinSize = 5pt, LineColor = black, FillColor = black]

\subsection{The graph $H_8$}

The graph $H_8$ is obtained from the graph $K_7$ by performing a single $\nabla Y$ move, resulting in the graph shown in Figure~\ref{F:H8}.

\begin{figure} [htbp]
$$\scalebox{.8}{\includegraphics{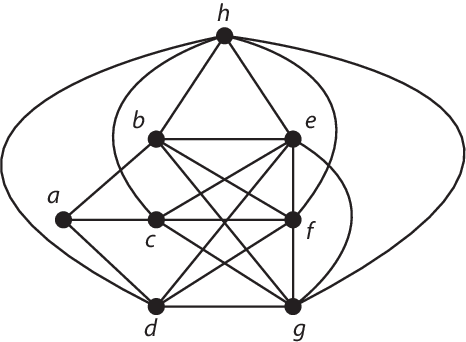}}$$
\caption{The graph $H_8$.}
\label{F:H8}
\end{figure}

The vertex $a$ has degree 3, the vertices $b, c, d$ have degree 5, and the vertices $e, f, g, h$ have degree 6.  So any automorphism of the graph must preserve these three sets of vertices (setwise).  Since the vertices $b, c, d$ all have the same neighbours, as do the vertices $e, f, g, h$, the vertices within each set can be permuted freely.  Hence the automorphism group of $H_8$ is isomorphic to $S_3 \times S_4$. If we let $\a = (bcd)$, $\b = (bc)$, $\g = (ef)$ and $\d = (fgh)$, then $\a$ and $\b$ generate $S_3$ with relations $\a^3 = \b^2 = (\a\b)^2 = 1$, and $\g$ and $\d$ generate $S_4$ with relations $\g^2 = \d^3 = (\g\d)^4 = 1$.

Table \ref{Ta:H8} lists a representative of each conjugacy class of nontrivial automorphisms, its order, a diagram of the fixed subgraph, whether the fixed subgraph embeds in $S^1$ or $S^2$, and (if one exists) a path between two interchanged vertices that contains no fixed vertices or flipped edges (as in the Path Lemma). If the Path Lemma does not apply, we indicate ``N/A'', with a superscript 1 if no pairs of vertices are interchanged, 2 if there are two or fewer fixed points, or 3 if the fixed subgraph does not embed in $S^2$. (In this last case, the Path Lemma might apply, but is redundant.)

\begin{table}
\begin{tabular}{|l|c|c|c|c|c|}
\hline
\rule{0pt}{.15in}Automorphism & Order & Fixed subgraph& $\hookrightarrow S^1$ & $\hookrightarrow S^2$ & Path Lemma \\ \hline
\rule{0pt}{.15in}$\a = (bcd)$ & 3 & \begin{tikzpicture}
\coordinate (O) at (0,0);
\NOEA(O){A} \NOWE(O){B} \SOEA(O){D}
\SOWE(O){C} 
  \WE[unit=1](O){G}

\Edges(A,B,C,D, A)
\Edges(A,C)
\Edges(B,D)

\end{tikzpicture}& No & Yes & N/A\textsuperscript{1} \\ \hline
\rule{0pt}{.15in}$\b = (bc)$ & 2 & \begin{tikzpicture} 
\Vertices{circle}{A,B,C,D,E}
  \Edges(A,B,C,D,E, A)
    \Edges(A,C,E,B,D, A)
     \WE[unit=.5](D){G}
        \Edges(D,G)
\end{tikzpicture}
& No & No & N/A\textsuperscript{3} \\ \hline
\rule{0pt}{.15in}$\g = (ef)$ & 2 &  \begin{tikzpicture}[rotate=0] 
  \Vertex{f}
  \WE[unit=1](f){i} \NO(f){e} \SO(f){g} \NO(i){h} \SO(i){d}

  \Edges(e,i,g,h,f,d,e)
  \Edges(e,h)  \Edges(i,f) \Edges(d,g)
      \tikzset{EdgeStyle/.style = {-
,bend right=60}}
\Edge(h)(d)
\end{tikzpicture} & No & No & N/A\textsuperscript{3} \\ \hline

\rule{0pt}{.15in}$\d = (fgh)$ & 3 &  \begin{tikzpicture}[rotate=0] 
  \Vertex{f}
  \EA(f){i} \WE(f){h} \NO(f){e}  \SO(f){g}   
  \Edges(h,f,i,e,h)
    \Edges(h,f,i,g,h)
\end{tikzpicture} & No & Yes & N/A\textsuperscript{1} \\ \hline
\rule{0pt}{.15in}$\g\d = (efgh)$ & 4 & \begin{tikzpicture}[rotate=0] 
  \Vertex{f}
  \EA(f){a} \WE(f){b} \SO(f){c}
   \Edges(c,b)
  \Edges(c,a)
    \Edges(c,f)
\end{tikzpicture}& No & Yes & N/A\textsuperscript{1} \\ \hline
\rule{0pt}{.15in}$(\g\d)^2 = (eg)(fh)$ & 2 & \begin{tikzpicture}[rotate=0] 
  \Vertex{f}
  \EA(f){a} \WE(f){b} \SO(f){c}
   \Edges(c,b)
  \Edges(c,a)
    \Edges(c,f)
\end{tikzpicture} & No & Yes & $e-f-g$ \\ \hline
\rule{0pt}{.15in}$\a\g = (bcd)(ef)$ & 6 & \begin{tikzpicture}[rotate=0] 
  \Vertex{f}
  \EA(f){i} \WE(f){h}  
  \Edges(f,i,f)
\end{tikzpicture}  & Yes & Yes & $e-d-f$\\ \hline
\rule{0pt}{.15in}$\a\d = (bcd)(fgh)$ & 3 &  \begin{tikzpicture}[rotate=0] 
  \Vertex{a}
  \EA(a){b} 
\end{tikzpicture} & Yes & Yes & N/A\textsuperscript{1} \\ \hline
\rule{0pt}{.15in}$\a\g\d = (bcd)(efgh)$ & 12 &  \begin{tikzpicture}[rotate=0] 
  \Vertex{a}
\end{tikzpicture} & Yes & Yes & N/A\textsuperscript{1} \\ \hline
\rule{0pt}{.15in}$\a(\g\d)^2 = (bcd)(eg)(fh)$ & 6 & \begin{tikzpicture}[rotate=0] 
  \Vertex{a}
\end{tikzpicture}& Yes & Yes & N/A\textsuperscript{2} \\ \hline
\rule{0pt}{.15in}$\b\g = (bc)(ef)$ & 2 &  \begin{tikzpicture}[rotate=0] 
  \Vertex{f}
  \EA(f){i} \WE(f){h} \NO(f){e}  
  \Edges(f,i,e,f)
    \Edges(f,h)
\end{tikzpicture}  & No & Yes & $b-e-c$ \\ \hline
\rule{0pt}{.15in}$\b\d = (bc)(fgh)$ & 6 & \begin{tikzpicture}[rotate=0] 
  \Vertex{a}
  \EA(a){b}    \WE(a){c} 
  \Edges(c,a,b)
\end{tikzpicture}  & Yes & Yes & $b-f-c$ \\ \hline
\rule{0pt}{.15in}$\b\g\d = (bc)(efgh)$ & 4 & \begin{tikzpicture}[rotate=0] 
  \Vertex{a}
  \EA(a){b} 
  \Edges(a,b)
\end{tikzpicture}  & Yes & Yes & $b-e-c$ \\ \hline
\rule{0pt}{.15in}$\b(\g\d)^2 = (bc)(eg)(fh)$ & 2 & \begin{tikzpicture}[rotate=0] 
  \Vertex{a}
  \EA(a){b} 
  \Edges(a,b)
\end{tikzpicture}  & Yes & Yes & $b-e-c$ \\ \hline
\end{tabular}
\caption{Automorphisms of $H_{8}.$}
\label{Ta:H8}
\end{table}

\begin{theorem}\label{T:H8}
A nontrivial group is positively realizable for $H_8$ if and only if it is $\Z_3$ or $\Z_2$.  Moreover, $H_8$ is intrinsically chiral, so these groups are only realizable by orientation-preserving homeomorphisms.
\end{theorem}
\begin{proof}
We can quickly see that the first 6 automorphisms in Table \ref{Ta:H8} are not realizable. They are not positively realizable, because their fixed subgraphs do not embed in $S^1$.  The automorphisms $\a, \d$ and $\g\d$ are not negatively realizable because they fix more than two points, and do not have order 2.  Automorphisms $\b$ and $\g$ are not negatively realizable because their fixed subgraphs do not embed in $S^2$, and $(\g\d)^2$ is not negatively realizable by the Path Lemma.

Of the remaining automorphisms, $\a\g, \b\g, \b\d, \b\g\d$ and $\b(\g\d)^2$ are not negatively realizable by the Path Lemma.  Automorphism $\a\d$ is not negatively realizable because it has odd order.  Since $(\a\g\d)^4 = \a$, which is not realizable, neither is $\a\g\d$.  Similarly, $[\a(\g\d)^2]^3 = (\g\d)^2$, so $\a(\g\d)^2$ is not realizable either.  So none of the automorphisms are negatively realizable, and hence $H_8$ is intrinsically chiral.

We already know that $\a, \b, \g, \d, \g\d, (\g\d)^2, \a\g\d$ and $\a(\g\d)^2$ are not positively realizable.  Also, $\b\g$ is not positively realizable by the Realizability Lemma.  Since $(\a\g)^4 = \a$, $(\b\d)^3 = \b$ and $(\b\g\d)^2 = (\g\d)^2$, we also know $\a\g$, $\b\d$ and $\b\g\d$ are not positively realizable.  This leaves only automorphisms conjugate to $\a\d$ and $\b(\g\d)^2$.

Suppose there is an embedding $\Gamma$ of $H_8$ such that automorphisms conjugate to $\a\d$ and $\b(\g\d)^2$ are realized by orientation-preserving homeomorphisms $h$ and $g$.  By the Isometry Theorem, we may assume $h$ and $g$ are isometries.  Since $\a\d$ and $\b(\g\d)^2$ both have fixed vertices, we know $h$ and $g$ are both rotations around great circles of $S^3$.  Since every automorphism of $H_8$ fixes vertex $a$, these great circles intersect at $a \in \Gamma$, so they must also intersect at the antipodal point. In addition, $h$ fixes one of $\{e, f, g, h\}$ and $g$ fixes one of $\{b, c, d\}$.  Without loss of generality, suppose $h$ fixes $e$ and $g$ fixes $b$.  Since $g$ has order 2, it takes the axis of $h$ to itself (setwise), hence $g(e)$ must be another vertex on the axis of $h$.  Since $h$ fixes only two vertices, and $g(e) \neq a$, we must have $g(e) = e$.  But this contradicts the fact that $g$ does not fix any of $\{e, f, g, h\}$.  So there is no embedding of $H_8$ that simultaneously realizes an automorphism conjugate to $\a\d$ and one conjugate to $\b(\g\d)^2$.  Moreover, the group generated by two automorphisms conjugate to $\a\d$ (that are not powers of each other) includes automorphisms conjugate to $\d$ or $(\g\d)^2$, which are not realizable; so we cannot simultaneously realize two automorphisms conjugate to $\a\d$.  Similarly, two automorphisms conjugate to $\b(\g\d)^2$ have a product conjugate to $\a$, $(\g\d)^2$ or $\a(\g\d)^2$, none of which are realizable.  So the topological symmetry group can at most positively realize $\Z_3$ (generated by an automorphism conjugate to $\a\d$) or $\Z_2$ (generated by an automorphism conjugate to $\b(\g\d)^2$).

The embeddings in Figure \ref{F:H8embeddings} positively realize $\Z_3$ and $\Z_2$.  On the left, the axis of order 3 is perpendicular to the page, and $e$ and $a$ are placed at antipodal points on the axis (in $S^3$).  On the right, the axis of rotation of order 2 is the vertical line through vertices $a$ and $b$.  So $\Z_3$ and $\Z_2$ are both positively realizable.
\end{proof}

\begin{figure} [htbp]
$$\scalebox{.7}{\includegraphics{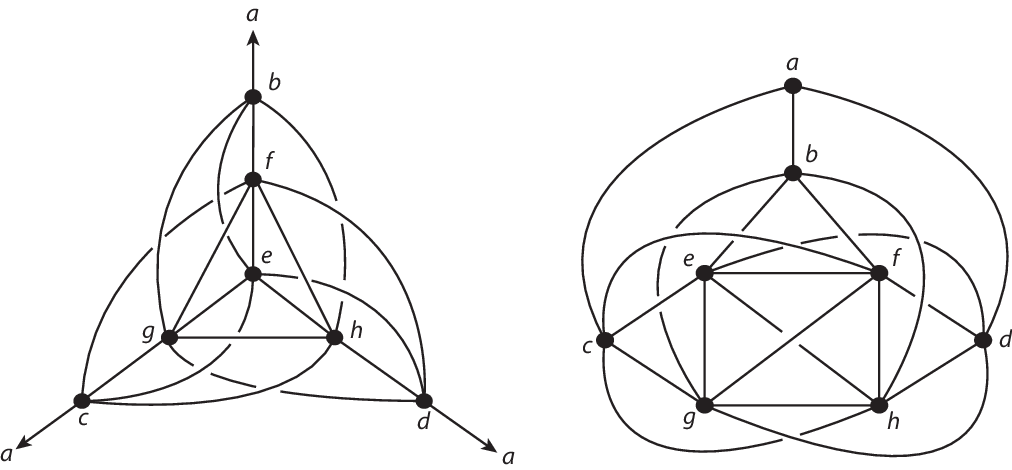}}$$
\caption{Embeddings of $H_8$ which positively realize $\Z_3$ (left) and $\Z_2$ (right).}
\label{F:H8embeddings}
\end{figure}

\subsection{The graph $H_9$}

The graph $H_9$ is obtained from $H_8$ (as shown in Figure \ref{F:H8}) by performing a $\nabla Y$-move to the triangle $efg$, resulting in the graph in Figure \ref{F:H9}.

\begin{figure} [htbp]
$$\scalebox{.8}{\includegraphics{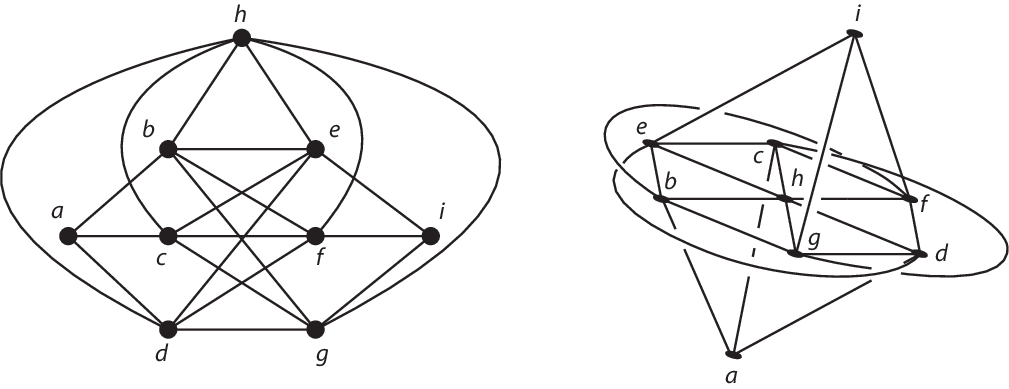}}$$
\caption{The graph $H_9$, and an embedding of $H_9$ which positively realizes $D_3$.}
\label{F:H9}
\end{figure}

Observe that if we delete vertices $a$ and $i$, the remaining graph is isomorphic to $K_{3,3,1}$.  Moreover, since $a$ and $i$ are the only vertices with degree 3, every automorphism of $H_9$ preserves this $K_{3,3,1}$ subgraph setwise, and the induced automorphism of $K_{3,3,1}$ determines the automorphism of $H_9$. Hence $\Aut(H_9) = \Aut(K_{3,3,1}) = \Aut(K_{3,3}) = (D_3 \times D_3) \semi \Z_2$ \cite{efmsw}. The automorphism group is generated by $\a = (bcd)$, $\b = (bc)$, $\g = (efg)$, $\d = (ef)$ and $\phi = (ai)(be)(cf)(dg)$ (note that every automorphism fixes vertex $h$).

Table \ref{Ta:H9} lists a representative of each conjugacy class of nontrivial automorphisms, along with the same information given for $H_8$ in Table \ref{Ta:H8}.  Note that in the diagrams of the fixed subgraph, empty circles represent the midpoints of edges whose endpoints are interchanged by the automorphism. This is based on the conjugacy classes of automorphisms of $K_{3,3}$ found by Nikkuni and Taniyama \cite{nt}.

\begin{table}
\begin{tabular}{|l|c|c|c|c|c|}
\hline
\rule{0pt}{.15in}Automorphism & Order & Fixed subgraph & $\hookrightarrow S^1$ & $\hookrightarrow S^2$ & Path Lemma \\ \hline
\rule{0pt}{.15in}$\a = (bcd)$ & 3 & \begin{tikzpicture}[rotate=0] 
  \Vertex{f}
  \EA(f){i} \WE(f){h} \NO(f){e} \SO(f){g}
  \WE(h){a}
  \Edges(e,i,g,h,e)
  \Edges(h,f,i)
\end{tikzpicture} & No & Yes & N/A\textsuperscript{1} \\ \hline
\rule{0pt}{.15in}$\b = (bc)$ & 2 & \begin{tikzpicture}[rotate=0] 
  \Vertex{f}
  \WE[unit=1](f){i} \NO(f){e} \SO(f){g} \NO(i){h} \SO(i){d}
  \WE(i){a}
  \Edges(e,i,g,h,f,d,e)
  \Edges(e,h) \Edges(a,d,g) \Edges(i,f)
\end{tikzpicture} & No & No & N/A\textsuperscript{3} \\ \hline
\rule{0pt}{.15in}$\a\g = (bcd)(efg)$ & 3 & \begin{tikzpicture}[rotate=0] 
  \Vertices{line}{a,h,i}
\end{tikzpicture} & Yes & Yes & N/A\textsuperscript{1} \\ \hline
\rule{0pt}{.15in}$\b\g = (bc)(efg)$ & 6 & \begin{tikzpicture}[rotate=0] 
  \Vertices{line}{a,d,h,i}
  \Edges(a,d,h)
\end{tikzpicture} & Yes & Yes & $b-e-c$\\ \hline
\rule{0pt}{.15in}$\b\d = (bc)(ef)$ & 2 & \begin{tikzpicture}[rotate=0] 
  \Vertices{line}{a,d,g,i}
  \NO(d){h}
  \Edges(a,d,g,i) \Edges(d,h,g)
\end{tikzpicture} & No & Yes & $b-e-c$ \\ \hline
\rule{0pt}{.15in}$\phi = (ai)(be)(cf)(dg)$ & 2 & \begin{tikzpicture}[rotate=0] 
  \Vertices{line}{h,x,y,z}
  \AddVertexColor{white}{x,y,z}
\end{tikzpicture} & Yes & Yes & $a-b-f-i$ \\ \hline
\rule{0pt}{.15in}$\b\phi = (ai)(becf)(dg)$ & 4 & \begin{tikzpicture}[rotate=0] 
  \Vertices{line}{h,x}
  \AddVertexColor{white}{x}
\end{tikzpicture} & Yes & Yes & N/A\textsuperscript{2} \\ \hline
\rule{0pt}{.15in}$\a\phi = (ai)(becfdg)$ & 6 & \begin{tikzpicture}[rotate=0] 
  \Vertex{h}
\end{tikzpicture} & Yes & Yes & N/A\textsuperscript{2} \\ \hline
\end{tabular}
\caption{Automorphisms of $H_{9}.$}
\label{Ta:H9}
\end{table}

By the Subgraph Lemma, given any embedding, the topological symmetry group for $H_9$ must be a subgroup of the topological symmetry group of the embedded $K_{3,3,1}$ subgraph.  The possible topological symmetry groups for $K_{3,3,1}$ were classified in \cite{efmsw}.

\begin{theorem} \label{T:K331} \cite{efmsw}
\begin{enumerate}
	\item A nontrivial group is positively realizable for  $K_{3,3,1}$ if and only if it is one of: $D_3$, $D_2$, $\mathbb{Z}_3$, and $\mathbb{Z}_2$.
	\item The groups that are realizable but {\bf not} positively realizable for $K_{3,3,1}$ are $D_4$ and $\Z_4$.
\end{enumerate}
\end{theorem}

We will use this result to classify the possible topological symmetry groups of $H_9$. 

\begin{theorem} \label{T:H9}
A nontrivial group is positively realizable for $H_9$ if and only if it is one of: $D_3$, $\mathbb{Z}_3$, and $\mathbb{Z}_2$. Moreover, $H_9$ is intrinsically chiral, so there are no groups that are realizable but not positively realizable.
\end{theorem}
\begin{proof}
We can quickly see from Table \ref{Ta:H9} that automorphisms conjugate to $\b$ and $\b\d$ are neither positively nor negatively realizable.  Since $(\b\phi)^2 = \b\d$, $\b\phi$ is also not realizable.  Also, $\a$ and $\a\g$ are not negatively realizable because they have odd order, and $\b\g$ and $\phi$ are not negatively realizable by the Path Lemma. Finally, $(\a\phi)^3 = (ai)(bf)(ed)(cg)$ is conjugate to $\phi$; since $\phi$ is not negatively realizable, neither is $\a\phi$.  Since none of the automorphisms are negatively realizable, $H_9$ is intrinsically chiral.

We already know $\b$, $\b\d$ and $\b\phi$ are not positively realizable; in addition, $\a$ is not positively realizable by the Realizability Lemma.  Since $(\b\g)^3 = \b$, $\b\g$ is also not positively realizable. Next, suppose $\a\phi$ is positively realizable; since it has a fixed point, by the Isometry Theorem it is realizable by a rotation about an axis containing $h$.  Then $(\a\phi)^2 = (bcd)(efg)$ is also a rotation, with the same axis.  But this means that $a$ and $i$ are also on the axis of rotation, and hence are fixed by $\a\phi$, which is a contradiction.  Hence $\a\phi$ is not positively realizable.

So the only automorphisms which could be positively realizable are conjugate to $\a\g$ or $\phi$. The embedding on the right in Figure \ref{F:H9} shows how to realize $D_3$ (based on an embedding for $K_{3,3,1}$ in \cite{efmsw}) by realizing both of these automorphisms together. Since the edge $\overline{ce}$ is not fixed by any symmetry of the embedding, the subgroups $\Z_3$ and $\Z_2$ are also realizable by the Subgroup Theorem.

By Theorem \ref{T:K331}, the only groups that could be positively realized for $H_9$ are $D_3$, $D_2$, $\mathbb{Z}_3$, and $\mathbb{Z}_2$. It only remains to show that $D_2$ is not positively realizable.  Suppose there is an embedding $\Gamma$ of $H_9$ with $D_2 \leq \TSG_+(\Gamma)$.  Then all three nontrivial elements in $D_2$ would be conjugate to $\phi$ (since these are the only positively realizable automorphisms of $H_9$ of order 2), and they would all commute with each other.  Without loss of generality, suppose one of the automorphisms in $D_2$ is $\phi$ itself. There are five other automorphisms in the conjugacy class of $\phi$; of these, three commute with $\phi$, but none of those three commute with each other.  So we cannot positively realize $D_2$.
\end{proof}

\subsection{The graph $H_{10}$}

The graph $H_{10}$ is obtained from $H_9$ (as shown in Figure \ref{F:H9}) by performing a $\nabla Y$-move on the triangle $beh$, resulting in the graph shown in Figure~\ref{F:H10}.

\begin{figure} [htbp]
$$\scalebox{.8}{\includegraphics{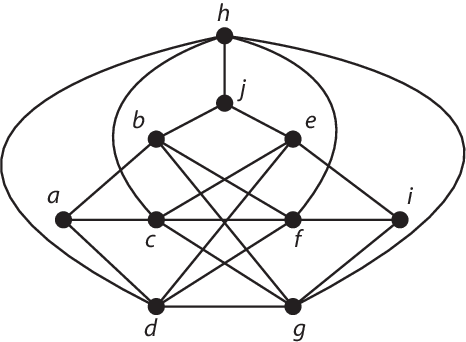}} \qquad \scalebox{.8}{\includegraphics{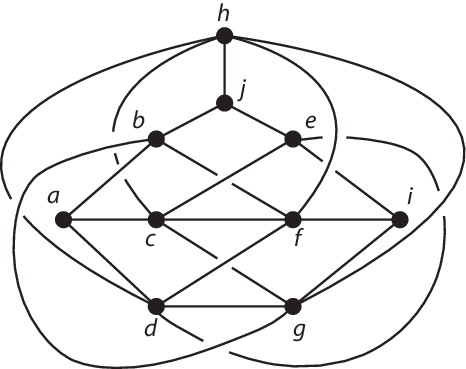}}$$
\caption{The graph $H_{10}$, and an embedding $\Gamma$ of $H_{10}$ with $\TSG(\Gamma) = \TSG_+(\Gamma) = \Z_2$.}
\label{F:H10}
\end{figure}

Any automorphism of $H_{10}$ must fix vertex $j$ (the only vertex of degree 3 adjacent to two vertices of degree 4), vertex $h$ (the only vertex of degree 5 {\em not} adjacent to a vertex of degree 4), the pair $\{a, i\}$ (the other vertices of degree 3), the pair $\{b, e\}$ (the vertices of degree 4), and the set $\{c, d, f, g\}$ (the other vertices of degree 5).  Moreover, the two pairs $\{c, d\}$ and $\{f, g\}$ are either preserved or interchanged, so the square $c-f-d-g-c$ is preserved by every automorphism.  Moreover, an automorphism of this square determines an automorphism of the entire graph, so $\Aut(H_{10}) = D_4$, generated by $\alpha = (ai)(be)(cfdg)$ and $\beta = (cd)$. The possible topological symmetry groups are $D_4$, $\Z_4$, $D_2$ and $\Z_2$ (the subgroups of $D_4$).

Table \ref{Ta:H10} lists a representative of each conjugacy class of nontrivial automorphisms, its order, a diagram of the fixed subgraph, whether the fixed subgraph embeds in $S^1$ or $S^2$, and (if one exists) a path between two interchanged vertices that contains no fixed vertices or flipped edges (as in the Path Lemma).

\begin{table}
\begin{tabular}{|l|c|c|c|c|c|}
\hline
\rule{0pt}{.15in}Automorphism & Order & Fixed subgraph & $\hookrightarrow S^1$ & $\hookrightarrow S^2$ & Path Lemma \\ \hline
\rule{0pt}{.15in}$\a = (ai)(be)(cfdg)$ & 4 & \begin{tikzpicture}[rotate=0] 
  \Vertex{h}
  \EA(h){j} 
  \Edges(h,j)
\end{tikzpicture}   & Yes & Yes & $a-b-f-i$ \\ \hline
\rule{0pt}{.15in}$\a^2 = (cd)(fg)$ & 2 &  \begin{tikzpicture}[rotate=0] 
  \Vertex{f}
  \EA(f){a} \WE(f){b} \SO(f){c}  \EA(a){d} \WE(b){e}
   \Edges(c,b)
  \Edges(c,a)
    \Edges(c,f)
     \Edges(a,d)
       \Edges(b,e)
\end{tikzpicture}  & No & Yes & $c-f-d$ \\ \hline
\rule{0pt}{.15in}$\b = (cd)$ & 2 & \begin{tikzpicture}[rotate=0] 
  \Vertex{f}
  \WE[unit=1](f){i} \NO(f){e} \SO(f){g} \NO(i){h} \SO(i){d}
  \WE(i){a}
  \Edges(e,i,g,h,f,d,e)
  \Edges(e,h) \Edges(a,d,g) \Edges(i,f)
        \tikzset{EdgeStyle/.style = {-
,bend left=60}}
\Edge(e)(g)
\end{tikzpicture} & No & No & N/A\textsuperscript{3}\\ \hline
\rule{0pt}{.15in}$\a\b = (ai)(be)(cg)(df)$ & 2 & \begin{tikzpicture}[rotate=0] 
  \Vertex{h}
  \EA(h){j} 
  \Edges(h,j)
\end{tikzpicture}  & Yes & Yes & $a-b-f-i$ \\ \hline
\end{tabular}
\caption{Automorphisms of $H_{10}.$}
\label{Ta:H10}
\end{table}

\begin{theorem} \label{T:H10}
The only nontrivial group that is positively realizable for $H_{10}$ is $\Z_2$.  Moreover, $H_{10}$ is intrinsically chiral, so no other groups are realizable.
\end{theorem}
\begin{proof}
From the table above, none of the automorphisms of $H_{10}$ are negatively realizable, either because the fixed subgraph does not embed in $S^2$ or due to the Path Lemma.  Hence $H_{10}$ is intrinsically chiral.

From the table, $\a^2$ and $\b$ are not positively realizable (since their fixed subgraphs do not embed in $S^1$), which means $\a$ is also not positively realizable.  So the only nontrivial automorphisms which could be positively realizable are those conjugate to $\a\b$, namely $\a\b$ and $\a^3\b$.  However, the product of these automorphisms is $\a^2$, so no embedding of $H_{10}$ can positively realize both of these automorphisms at the same time.  Hence, $\Z_2$ is the only group which can possibly be positively realized.

The embedding $\Gamma$ of $H_{10}$ in Figure \ref{F:H10} is symmetric under the half-turn rotation around the axis through the edge $\overline{hj}$.  Since no other nontrivial groups are realizable, this means $\TSG(\Gamma) = \TSG_+(\Gamma) = \Z_2$.
\end{proof}

\subsection{The graph $H_{11}$}

The graph $H_{11}$ is obtained from $H_{10}$ (as shown in Figure \ref{F:H10}) by performing a $\nabla Y$-move on the triangle $hcf$, resulting in the graph shown in Figure~\ref{F:H11}.

\begin{figure} [htbp]
$$\scalebox{.8}{\includegraphics{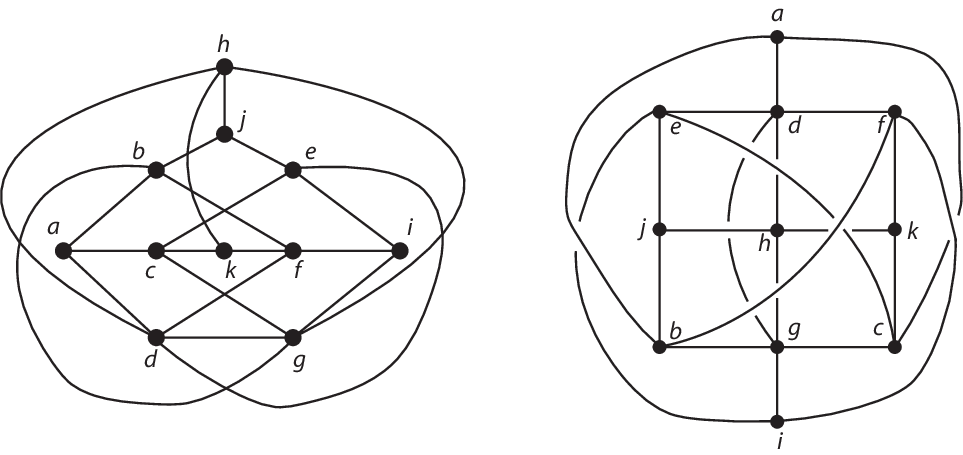}}$$
\caption{The graph $H_{11}$, and an embedding $\Gamma$ of $H_{11}$ with $\TSG(\Gamma) = \TSG_+(\Gamma) = \Z_2$.}
\label{F:H11}
\end{figure}

Every automorphism of $H_{11}$ fixes vertex $h$ (the only vertex of degree 4 adjacent to two vertices of degree 3). Similarly, looking at the degrees of the vertices and their neighbors, every automorphism preserves each of the sets $\{a,i\}$, $\{j, k\}$, $\{d, g\}$ and $\{b, c, e, f\}$. In fact, the permutations of the sets $\{a, i\}$ and $\{j, k\}$ determine the permutation of the other vertices.  Hence $\Aut(H_{11}) = \Z_2 \times \Z_2 = D_2$, and the possible nontrivial permutations of the vertices are: $\alpha = (ai)(be)(cf)(dg)$, $\beta = (jk)(bc)(ef)$, and $\alpha\beta = (ai)(jk)(bf)(ce)(dg)$. The properties of these automorphisms are listed in Table \ref{Ta:H11}.

\begin{table}
\begin{tabular}{|l|c|c|c|c|c|}
\hline
\rule{0pt}{.15in}Automorphism & Order & Fixed subgraph & $\hookrightarrow S^1$ & $\hookrightarrow S^2$ & Path Lemma \\ \hline
\rule{0pt}{.15in}$\alpha = (ai)(be)(cf)(dg)$ & 2 &  \begin{tikzpicture}[rotate=0] 
  \Vertex{a}
  \EA(a){b}    \WE(a){c}   \EA(b){d} 
  \Edges(c,a,b)
   \AddVertexColor{white}{d}
\end{tikzpicture}  & Yes & Yes & $a-b-f-i$ \\ \hline
\rule{0pt}{.15in}$\beta = (jk)(bc)(ef)$ & 2 &  \begin{tikzpicture}[rotate=0] 
  \Vertices{line}{a,d,g,i}
  \NO(d){h}
  \Edges(a,d,g,i) \Edges(d,h,g)
\end{tikzpicture} & No & Yes & $j-b-f-k$ \\ \hline
\rule{0pt}{.15in}$\alpha\beta = (ai)(jk)(bf)(ce)(dg)$ & 2 &  \begin{tikzpicture}[rotate=0] 
  \Vertices{line}{h,x,y,z}
  \AddVertexColor{white}{x,y,z}
\end{tikzpicture}  & Yes & Yes & $a-b-j-e-i$ \\ \hline
\end{tabular}
\caption{Automorphisms of $H_{11}.$}
\label{Ta:H11}
\end{table}

\begin{theorem} \label{T:H11}
The only nontrivial group that is positively realizable for $H_{11}$ is $\Z_2$.  Moreover, $H_{11}$ is intrinsically chiral, so no other groups are realizable.
\end{theorem}
\begin{proof}
We can see from Table \ref{Ta:H11} that none of the automorphisms of $H_{11}$ are negatively realizable, due to the Path Lemma, so $H_{11}$ is intrinsically chiral.  We also observe that $\beta$ is not positively realizable, so the largest possible positively realizable group is $\Z_2$ (note that if $\alpha$ and $\alpha\beta$ were both positively realizable for one embedding, then $\beta$ would also be positively realizable, so only one automorphism can be realized at once).

The embedding $\Gamma$ of $H_{11}$ in Figure \ref{F:H11} is symmetric under the half-turn rotation around the point $h$, positively realizing $\alpha\beta$.  Here we imagine all the vertices are in the plane of the page, the edge $\overline{dg}$ is beneath the plane, with its midpoint directly below vertex $h$, and the edges $\overline{ec}$ and $\overline{bf}$ are above the plane, with their midpoints directly above vertex $h$ (and with $\overline{bf}$ above $\overline{ec}$).  Since no other nontrivial groups are realizable, this means $\TSG(\Gamma) = \TSG_+(\Gamma) = \Z_2$.
\end{proof}

\subsection{The graph $H_{12}$}

The graph $H_{12}$ is obtained from $H_{11}$ (as shown in Figure \ref{F:H11}) by performing a $\nabla Y$-move on the triangle $hdg$, resulting in the graph shown in Figure~\ref{F:H12}.

\begin{figure} [htbp]
$$\scalebox{.8}{\includegraphics{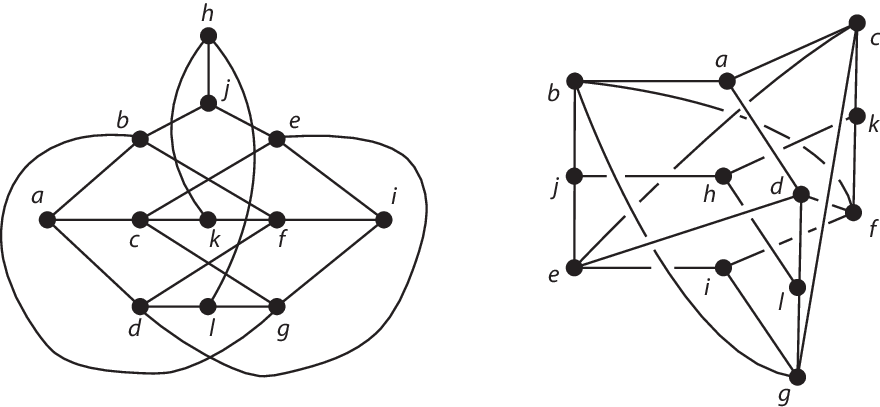}}$$
\caption{The graph $H_{12}$, and an embedding $\Gamma$ of $H_{12}$ with $\TSG(\Gamma) = \TSG_+(\Gamma) = D_3$.}
\label{F:H12}
\end{figure}

Note that any automorphism of $H_{12}$ fixes $h$ (the only vertex of degree 3 all of whose neighbors are also degree 3), so it permutes $\{j, k, l\}$, and also permutes $\{a, i\}$.  The result determines the permutation of the remaining 6 vertices. So $\Aut(H_{12}) = D_3 \times \Z_2$, generated by $\alpha = (jkl)(bcd)(efg)$, $\beta = (ai)(jk)(bf)(ec)(dg)$ and $\gamma = (ai)(be)(cf)(dg)$ ($\alpha$ and $\beta$ generate a subgroup isomorphic to $D_3$, and $\gamma$ commutes with both). The properties of the automorphisms are given in Table \ref{Ta:H12} (we list one automorphism from each conjugacy class).

\begin{table}
\begin{tabular}{|l|c|c|c|c|c|}
\hline
\rule{0pt}{.15in}Automorphism & Order & Fixed subgraph & $\hookrightarrow S^1$ & $\hookrightarrow S^2$ & Path Lemma \\ \hline
\rule{0pt}{.15in}$\a = (jkl)(bcd)(efg)$ & 3 & \begin{tikzpicture}[rotate=0] 
  \Vertices{line}{a,h,i}
\end{tikzpicture}  & Yes & Yes & N/A\textsuperscript{1} \\ \hline
\rule{0pt}{.15in}$\b = (ai)(jk)(bf)(ec)(dg)$ & 2 & \begin{tikzpicture}[rotate=0] 
  \Vertex{a}
  \EA(a){b}    \EA(b){d}  \EA(d){e} 
  \Edges(a,b)
   \AddVertexColor{white}{d,e}
\end{tikzpicture}  & Yes & Yes & $a-c-g-i$ \\ \hline
\rule{0pt}{.15in}$\g = (ai)(be)(cf)(dg)$ & 2 &  \begin{tikzpicture}[rotate=0] 
  \Vertex{f}
  \EA(f){a} \WE(f){b} \SO(f){c}
   \Edges(c,b)
  \Edges(c,a)
    \Edges(c,f)
\end{tikzpicture}  & No & Yes & $a-b-f-i$ \\ \hline
\rule{0pt}{.15in}$\a\g = (ai)(bfdecg)(jkl)$ & 6 & \begin{tikzpicture}[rotate=0] 
  \Vertex{h}
\end{tikzpicture}& Yes & Yes &  N/A\textsuperscript{2} \\ \hline
\rule{0pt}{.15in}$\b\g = (bc)(ef)(jk)$ & 2 &   \begin{tikzpicture}[rotate=0] 
  \Vertex{f}
  \EA(f){a} \WE(f){b} \SO(f){c}  \EA(a){d} \WE(b){e}
   \Edges(c,b)
  \Edges(c,a)
    \Edges(c,f)
     \Edges(a,d)
       \Edges(b,e)
\end{tikzpicture}  & No & Yes & $j-b-f-k$ \\ \hline
\end{tabular}
\caption{Automorphisms of $H_{12}$.}
\label{Ta:H12}
\end{table}

\begin{theorem} \label{T:H12}
The groups $D_3$, $\Z_3$ and $\Z_2$ are positively realizable for $H_{12}$.  Moreover, $H_{12}$ is intrinsically chiral, and there are no other nontrivial groups that are realizable for $H_{12}$.
\end{theorem}
\begin{proof}
From Table \ref{Ta:H12}, $\g$, $\b\g$ and their conjugates are neither positively nor negatively realizable.  Since $(\a\g)^3 = \g$, $\a\g$ and its conjugates are not realizable either.  The remaining automorphisms are not negatively realizable (either by the Path Lemma or because they have odd order), so $H_{12}$ is intrinsically chiral. Since only automorphisms conjugate to $\a$ or $\b$ can be realizable (namely, $\a, \a^2, \b, \a\b$ and $\a^2\b$), the only possible topological symmetry groups are subgroups of $D_3$: $D_3$, $\Z_3$ and $\Z_2$.

Figure~\ref{F:H12} shows an embedding of $H_{12}$ which positively realizes $D_3$. Here $\alpha$ is realized by the rotation of order 3 around the axis through $a$, $h$ and $i$, and $\beta$ is realized by the rotation of order 2 around the axis through the edge $\overline{hl}$.  Since the edge $\overline{ab}$ is not fixed by any nontrivial element of the topological symmetry group, we can also positively realize $\Z_3$ and $\Z_2$ by the Subgroup Theorem.
\end{proof}

\subsection{The graph $F_{9}$}

The graph $F_{9}$ is obtained from $H_{8}$ (as shown in Figure \ref{F:H8}) by performing a $\nabla Y$-move on the triangle $dfg$, resulting in the graph shown in Figure~\ref{F:F9}.

\begin{figure} [htbp]
$$\scalebox{.8}{\includegraphics{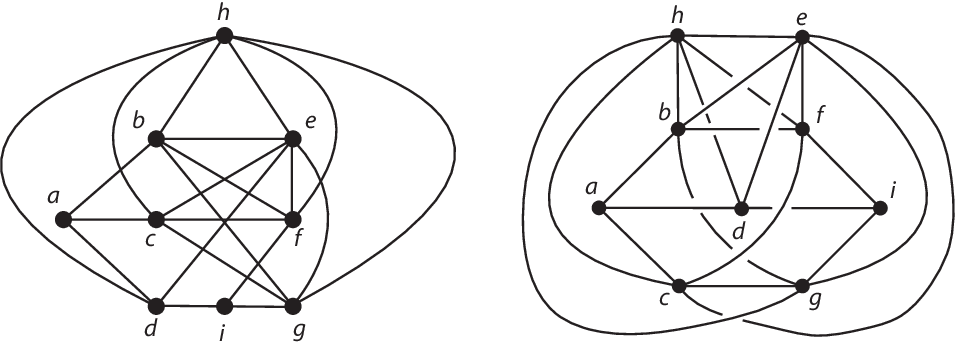}}$$
\caption{The graph $F_{9}$, and an embedding $\Gamma$ of $F_{9}$ with $\TSG(\Gamma) = \TSG_+(\Gamma) = \Z_2$.}
\label{F:F9}
\end{figure}

Every automorphism of $F_9$ fixes $d$ (the only vertex of degree 4), $\{a, i\}$ (the vertices of degree 3), $\{h, e\}$ (the vertices of degree 6) and $\{b, c, f, g\}$ (the vertices of degree 5).  Note that the subgraph of $F_9$ induced by $\{b,c,f,g\}$ is the 4-cycle (or square) $bfcg$.  So every automorphism of $F_9$ must preserve this square.  A permutation of $\{b, c, f, g\}$ determines a permutation of $\{a, i\}$; since $h$ and $e$ are both adjacent to every vertex in the square (along with $d$ and each other), they can be interchanged independent of the permutation of $\{b, c, f, g\}$.  So $\Aut(F_9) \cong D_4 \times \Z_2$, generated by $\alpha = (bfcg)(ai)$, $\beta = (bc)$ and $\gamma = (eh)$. The properties of the automorphisms are given in Table \ref{Ta:F9} (we list one automorphism from each conjugacy class).

\begin{table}
\begin{tabular}{|l|c|c|c|c|c|}
\hline
\rule{0pt}{.15in}Automorphism & Order & Fixed subgraph & $\hookrightarrow S^1$ & $\hookrightarrow S^2$ & Path Lemma \\ \hline
\rule{0pt}{.15in}$\a = (bfcg)(ai)$ & 4 &  \begin{tikzpicture}[rotate=0] 
  \Vertex{k}
  \NO(k){j} \EA(k){l}
  \Edges(j,k,l,j)
\end{tikzpicture} & Yes & Yes & $a-b-f-i$ \\ \hline
\rule{0pt}{.15in}$\a^2 = (bc)(fg)$ & 2 &  \begin{tikzpicture}[rotate=0] 
  \Vertex{k}
  \NO(k){j} \EA(k){l} \NOEA(k){a} \NOWE(k){b}
  \Edges(j,k,l,j)   \Edges(j,a) \Edges(j,b)
\end{tikzpicture}  & No & Yes & $b-f-c$ \\ \hline
\rule{0pt}{.15in}$\b = (bc)$ & 2 & \begin{tikzpicture}[rotate=0] 
  \Vertex{f}
  \WE[unit=1](f){i} \NO(f){e} \SO(f){g} \NO(i){h} \SO(i){d}
  \WE(i){a}
  \Edges(e,i,g,h,f,d,e)
  \Edges(e,h) \Edges(a,d,g) \Edges(i,f) \Edges(e,f)
\end{tikzpicture} & No & No & N/A\textsuperscript{3} \\ \hline
\rule{0pt}{.15in}$\a\b = (bg)(cf)(ai)$ & 2 & \begin{tikzpicture}[rotate=0] 
   \Vertex{a}
  \EA(a){b} \NO(a){c}    \EA(b){d}    \EA(d){e}  
  \Edges(a,b,c,a)
     \AddVertexColor{white}{d,e}
\end{tikzpicture} & No & Yes & $a-c-g-i$ \\ \hline
\rule{0pt}{.15in}$\g = (eh)$ & 2 & \begin{tikzpicture}[rotate=0] 
  \Vertex{f}
  \WE[unit=1](f){i} \NO(f){e} \SO(f){g} \NO(i){h} \SO(i){d} \EA[unit=.5](g){j} 
  \WE[unit=.5](g){k} 
  \Edges(e,d,g,h,f,d,e, i, g)
  \Edges(e,h) \Edges(a,d,g) \Edges(i,f) \Edges(i,g)
      \AddVertexColor{white}{j}
\end{tikzpicture} & No & No & N/A\textsuperscript{3} \\ \hline
\rule{0pt}{.15in}$\a\g = (bfcg)(ai)(eh)$ & 4 &  \begin{tikzpicture}[rotate=0] 
     \Vertices{line}{a,b}
    \AddVertexColor{white}{b}
  
\end{tikzpicture}  & Yes & Yes & N/A\textsuperscript{2} \\ \hline
\rule{0pt}{.15in}$\a^2\g = (bc)(fg)(eh)$ & 2 & \begin{tikzpicture}[rotate=0] 
     \Vertices{line}{a,b,c}
  \Edges(a,b,c)
    \EA(c){d}   
    \AddVertexColor{white}{d}
  
\end{tikzpicture} & Yes & Yes & $b-f-c$ \\ \hline
\rule{0pt}{.15in}$\b\g = (bc)(eh)$ & 2 &  \begin{tikzpicture}[rotate=0] 
     \Vertices{line}{a,b,c,d, e}
  \Edges(a,b,c, d)
    \NO(c){f}   
      \Edges(c,f)
    \AddVertexColor{white}{e}
  
\end{tikzpicture}  & No & Yes & $b-e-c$ \\ \hline
\rule{0pt}{.15in}$\a\b\g = (bg)(cf)(ai)(eh)$ & 2 &  \begin{tikzpicture}[rotate=0] 
  \Vertices{line}{h,x,y,z}
  \AddVertexColor{white}{x,y,z}
\end{tikzpicture} & Yes & Yes & $b-f-i-g$ \\ \hline
\end{tabular}
\caption{Automorphisms of $F_{9}$.}
\label{Ta:F9}
\end{table}

\begin{theorem} \label{T:F9}
The only nontrivial group that is positively realizable for $F_{9}$ is $\Z_2$. Moreover, $F_9$ is intrinsically chiral, so no other groups are realizable.
\end{theorem}
\begin{proof}
We first observe from Table \ref{Ta:F9} that $\b$ and $\g$ are not realizable by the Realizability Lemma.  Moreover, $\a^2$ is not positively realizable by the Realizability Lemma, and is not negatively realizable by the Path Lemma; similarly, $\a\b$ and $\b\g$ are not realizable.  Since $(\a\g)^2 = \a^2$, neither $\a\g$ nor $\a$ are realizable either.  Finally, we see that $\a^2\g$ and $\a\b\g$ are not negatively realizable by the Path Lemma.  So none of the automorphisms are negatively realizable, and hence $F_9$ is intrinsically chiral.

We already know $\a$, $\a^2$, $\b$, $\a\b$, $\g$, $\a\g$ and $\b\g$ (and their conjugates) are not positively realizable. We are left with three automorphisms that may be positively realizable: $\alpha\beta\gamma$, its conjugate $\a^3\b\g$, and $\alpha^2\gamma$ (which is only conjugate to itself), all of order 2.  The product of any two of these yields an unrealizable automorphism, so the only group which may be realizable is $\Z_2$.  Figure \ref{F:F9} shows an embedding of $F_9$ which positively realizes $\Z_2$ (specifically, it realizes the automorphism $\alpha\beta\gamma$ by a $180^\circ$ rotation about the vertical axis). 
\end{proof}

\subsection{The graph $F_{10}$}

The graph $F_{10}$ is obtained from $F_{9}$ (as shown in Figure \ref{F:F9}) by performing a $\nabla Y$-move on the triangle $deh$, resulting in the graph shown in Figure~\ref{F:F10}.

\begin{figure} [htbp]
$$\scalebox{.8}{\includegraphics{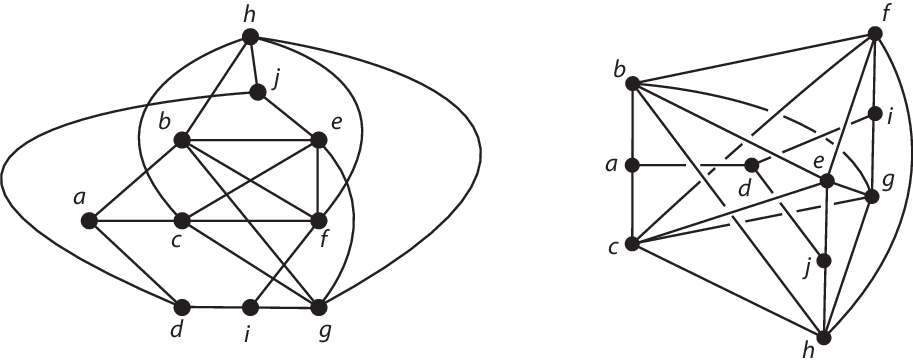}}$$
\caption{The graph $F_{10}$, and an embedding $\Gamma$ of $F_{10}$ with $\TSG(\Gamma) = \TSG_+(\Gamma) = D_3$.}
\label{F:F10}
\end{figure}

Every automorphism of $F_{10}$ fixes $d$ (the only vertex of degree 3 that is adjacent to the other three vertices of degree 3), $\{a, i, j\}$ (the other vertices of degree 3), and $\{b, c, e, h, f, g\}$ (the vertices of degree 5).  Moreover, $a$ is adjacent to $\{b, c\}$, $i$ is adjacent to $\{f, g\}$ and $j$ is adjacent to $\{e,h\}$; so the permutation of $\{a, i, j\}$ determines a permutation of these three pairs of vertices.  Finally, the vertices within each pair can be interchanged (since they have the same neighbours).  So the group of automorphisms is generated by $\a = (aij)(bfe)(cgh)$, $\b = (ai)(bf)(cg)$, $\gamma_a = (bc)$, $\gamma_i = (fg)$ and $\gamma_j = (eh)$.  The automorphisms $\alpha$ and $\beta$ generate a subgroup isomorphic to $D_3$, while $\gamma_a, \gamma_i,$ and $\gamma_j$ generate $(\Z_2)^3$.  The automorphism group is a semidirect product of these two subgroups, with the homeomorphism $\phi: D_3 \rightarrow \Aut((\Z_2)^3)$ defined by $\phi(g)(\gamma_x) = g\gamma_x g^{-1} = \gamma_{g(x)}$. So $\Aut(F_{10}) \cong (\Z_2)^3 \semi D_3$. The 47 nontrivial automorphisms fall into 9 conjugacy classes (where $x, y \in \{a, i, j\}$):
\begin{itemize}
	\item $[\a] = \{\a, \a^2, \a\g_a\g_i, \a\g_a\g_j, \a\g_i\g_j, \a^2\g_a\g_i, \a^2\g_a\g_j, \a^2\g_i\g_j\}$
	\item $[\a\g_a] = \{\a\g_a, \a\g_i, \a\g_j, \a^2\g_a, \a^2\g_i, \a^2\g_j, \a\g_a\g_i\g_j, \a^2\g_a\g_i\g_j\}$
	\item $[\b] = \{\b, \a\b, \a^2\b, \b\g_a\g_i, \a\b\g_a\g_j, \a^2\b\g_i\g_j\}$
	\item $[\b\g_a] = \{\b\g_a, \b\g_i, \a\b\g_a, \a\b\g_j, \a^2\b\g_i, \a^2\b\g_j\}$
	\item $[\b\g_j] = \{\b\g_j, \a\b\g_i, \a^2\b\g_a, \b\g_a\g_i\g_j, \a\b\g_a\g_i\g_j, \a^2\b\g_a\g_i\g_j\}$
	\item $[\b\g_a\g_j] = \{\b\g_a\g_j, \b\g_i\g_j, \a\b\g_a\g_i, \a\b\g_i\g_j, \a^2\b\g_a\g_i, \a^2\b\g_a\g_j\}$
	\item $[\g_a] = \{\g_a, \g_i, \g_j\}$
	\item $[\g_a\g_i] = \{\g_a\g_i, \g_a\g_j, \g_i\g_j\}$
	\item $[\g_a\g_i\g_j] = \{\g_a\g_i\g_j\}$
\end{itemize}

The properties of the automorphisms are given in Table \ref{Ta:F10} (we list one automorphism from each conjugacy class).

\begin{table}
\begin{tabular}{|l|c|c|c|c|c|}
\hline
\rule{0pt}{.15in}Automorphism & Order & Fixed subgraph & $\hookrightarrow S^1$ & $\hookrightarrow S^2$ & Path Lemma \\ \hline
\rule{0pt}{.15in}$\a = (aij)(bfe)(cgh)$ & 3 & \begin{tikzpicture}[rotate=0] 
  \Vertex{d}
\end{tikzpicture} & Yes & Yes & N/A\textsuperscript{1} \\ \hline
\rule{0pt}{.15in}$\a\g_a = (aij)(bghcfe)$ & 6 & \begin{tikzpicture}[rotate=0] 
  \Vertex{d}
\end{tikzpicture} & Yes & Yes & N/A\textsuperscript{1} \\ \hline
\rule{0pt}{.15in}$\b = (ai)(bf)(cg)$ & 2 &  \begin{tikzpicture}[rotate=0] 
  \Vertex{f}
  \EA(f){a} \WE(f){b} \SO(f){c}
   \Edges(c,b)
  \Edges(c,a)
    \Edges(c,f)
   \EA[unit=1](c){x}  \EA[unit=.5](x){y} 
  \AddVertexColor{white}{x,y}
  
\end{tikzpicture} & No & Yes & $a-b-g-i$ \\ \hline
\rule{0pt}{.15in}$\b\g_a = (ai)(bgcf)$ & 4 & \begin{tikzpicture}[rotate=0] 
  \Vertex{f}
  \EA(f){a} \WE(f){b} \SO(f){c}
   \Edges(c,b)
  \Edges(c,a)
    \Edges(c,f)
  
\end{tikzpicture} & No & Yes & $a-b-g-i$ \\ \hline
\rule{0pt}{.15in}$\b\g_j = (ai)(bf)(cg)(eh)$ & 2 &  \begin{tikzpicture}[rotate=0] 
  \Vertex{a}
  \EA(a){b}    \EA(b){d}  \EA(d){e} 
  \Edges(a,b)
   \AddVertexColor{white}{d,e}
\end{tikzpicture}& Yes & Yes & $a-b-g-i$ \\ \hline
\rule{0pt}{.15in}$\b\g_a\g_j = (ai)(bgcf)(eh)$ & 4 &  \begin{tikzpicture}[rotate=0] 
  \Vertex{a}
  \EA(a){b}     \Edges(a,b)
\end{tikzpicture} & Yes & Yes & $a-b-g-i$ \\ \hline
\rule{0pt}{.15in}$\g_a = (bc)$ & 2 & \begin{tikzpicture}[rotate=0] 
  \Vertex{f}
  \WE[unit=1](f){i} \NO(f){e} \SO(f){g} \NO(i){h} \SO(i){d} \EA[unit=.5](g){j} 
  \WE[unit=.5](g){k} 
  \Edges(h,g,k,d,f,h,e,i,f,d,e)
  \Edges(i,g) \Edges(a,d,g) \Edges(i,f) 
      \tikzset{EdgeStyle/.style = {-
,bend right=60}}
\Edge(k)(j)  
\end{tikzpicture}

 & No & No & N/A\textsuperscript{3} \\ \hline
\rule{0pt}{.15in}$\g_a\g_i = (bc)(fg)$ & 2 & \begin{tikzpicture}[rotate=0] 
\Vertex{A}
\EA(A){B} 
\NOEA(B){F} \NOWE(A){G} \SOEA(B){H} \SOWE(A){I}
  \Edges(G,A,B,F)   \Edges(I,A,B,H)
\end{tikzpicture}  & No & Yes & $b-f-c$ \\ \hline
\rule{0pt}{.15in}$\g_a\g_i\g_j = (bc)(fg)(eh)$ & 2 & \begin{tikzpicture}[rotate=0] 
  
    \Vertex{f}
  \EA(f){a} \WE(f){b} \SO(f){c}
   \Edges(c,b)
  \Edges(c,a)
    \Edges(c,f)
   
\end{tikzpicture} 

  & No & Yes & $b-f-c$ \\ \hline
\end{tabular}
\caption{Automorphisms of $F_{10}$.}
\label{Ta:F10}
\end{table}

\begin{theorem} \label{T:F10}
The only nontrivial group that is positively realizable for $F_{10}$ is $\Z_3$.  Moreover, $F_{10}$ is intrinsically chiral, so no other nontrivial group is realizable.
\end{theorem}
\begin{proof}
From Table \ref{Ta:F10}, we easily see that none of $\b$, $\b\g_a$, $\g_a$, $\g_a\g_i$ and $\g_a\g_i\g_j$ (and their conjugates) are positively or negatively realizable (by the Realizability Lemma and the Path Lemma).  Also, $\a\g_a$ is not realizable, since $(\a\g_a)^3 = \g_a\g_i\g_j$, which is not realizable.  Similarly, $(\b\g_a\g_j)^2 = \g_a\g_i$, so $\b\g_a\g_j$ is not realizable.  Finally, $\b\g_j$ is not negatively realizable by the Path Lemma, and $\a$ is not negatively realizable since it has odd order.  So no automorphisms are negatively realizable, and $F_{10}$ is intrinsically chiral.

The automorphisms conjugate to $\a$ and $\b\g_j$ are positively realizable. The embedding on the right in Figure \ref{F:F10} realizes both $\a$ (by a rotation of order 3 around the vertical axis through $d$) and $\b\g_a\g_i\g_j = (ai)(bg)(cf)(eh)$ (by a rotation of order 2 around the axis through $\overline{dj}$); together, these generate a group isomorphic to $D_3$. However, including any other automorphism conjugate to either $\a$ or $\b\g_a\g_i\g_j$ (that is not already in the group) generates an automorphism in one of the other conjugacy classes, which are not realizable.  So this embedding positively realizes $D_3$.  Since the edge $\overline{ab}$ (among others) is not pointwise fixed by any automorphism, the subgroups $\Z_3$ and $\Z_2$ are also positively realizable by the Subgroup Theorem.

Any combination of automorphisms from the conjugacy classes of $\a$ and $\b\g_j$ which do not generate $D_3$, $\Z_3$ or $\Z_2$ generate groups containing non-realizable automorphisms conjugate to either $\b\g_a$ or $\g_a\g_i$, so no other groups are realizable.
\end{proof}

\subsection{The graph $E_{10}$}

The graph $E_{10}$ is obtained from $F_{9}$ (as shown in Figure \ref{F:F9}) by performing a $\nabla Y$-move on the triangle $bef$, resulting in the graph shown in Figure~\ref{F:E10}.

\begin{figure} [htbp]
$$\scalebox{.8}{\includegraphics{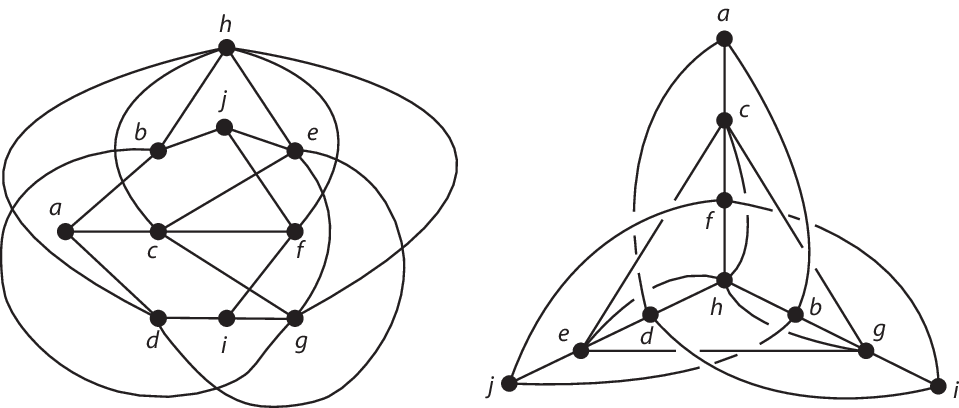}}$$
\caption{The graph $E_{10}$, and an embedding $\Gamma$ of $E_{10}$ with $\TSG(\Gamma) = \TSG_+(\Gamma) = \Z_3$.}
\label{F:E10}
\end{figure}

Any automorphism of $E_{10}$ will fix vertex $h$ (the only vertex of degree 6), and will permute the sets $\{a, i, j\}$ of degree 3 vertices, $\{b, d, f\}$ of degree 4 vertices, and $\{c, e, g\}$ of degree 5 vertices.  Moreover, since each of $c, e, g$ is adjacent to exactly one vertex of degree 3 and one of degree 4, the permutation of $\{c, e, g\}$ determines the permutation of the remaining vertices.  Hence $\Aut(E_{10}) \cong D_3$, generated by $\alpha = (ceg)(aji)(fdb)$ and $\beta = (ce)(aj)(df)$. The properties of the automorphisms are given in Table \ref{Ta:E10} (every nontrivial automorphism is conjugate to $\a$ or $\b$).

\begin{table}
\begin{tabular}{|l|c|c|c|c|c|}
\hline
\rule{0pt}{.15in}Automorphism & Order & Fixed subgraph & $\hookrightarrow S^1$ & $\hookrightarrow S^2$ & Path Lemma \\ \hline
\rule{0pt}{.15in}$\a = (ceg)(aji)(fdb)$ & 3 & \begin{tikzpicture}[rotate=0] 
  \Vertex{d}
\end{tikzpicture} & Yes & Yes & N/A\textsuperscript{1} \\ \hline
\rule{0pt}{.15in}$\b = (ce)(aj)(df)$ & 2 &  \begin{tikzpicture}[rotate=0] 
   \Vertex{a}
  \EA(a){b} \NO(a){c}    \EA(b){d}    \EA(d){e}  
  \Edges(a,b,c, a)   \Edges(b,d)

     \AddVertexColor{white}{e}
\end{tikzpicture} & No & Yes & $a-c-f-j$ \\ \hline
\end{tabular}
\caption{Automorphisms of $E_{10}$.}
\label{Ta:E10}
\end{table}

\begin{theorem} \label{T:E10}
The only nontrivial group that is positively realizable for $E_{10}$ is $\Z_3$.  Moreover, $E_{10}$ is intrinsically chiral, so no other nontrivial group is realizable.
\end{theorem}
\begin{proof}
From Table \ref{Ta:E10}, $\beta$ is not positively realizable (by the Realizability Lemma) or negatively realizable (by the Path Lemma).  Also, $\a$ is not negatively realizable, since it has odd order, showing that $E_{10}$ is intrinsically chiral.

Figure \ref{F:E10} shows an embedding of $E_{10}$ which realizes the automorphism $\alpha$ by the rotation around the vertex $h$.  Since the only realizable automorphisms are the powers of $\a$, this shows $\Z_3$ is positively realizable, but no other nontrivial group is realizable.
\end{proof}

\subsection{The graph $E_{11}$}

The graph $E_{11}$ is obtained from $E_{10}$ (as shown in Figure \ref{F:E10}) by performing a $\nabla Y$-move on the triangle $cfh$, resulting in the graph shown in Figure~\ref{F:E11}.

\begin{figure} [htbp]
$$\scalebox{.8}{\includegraphics{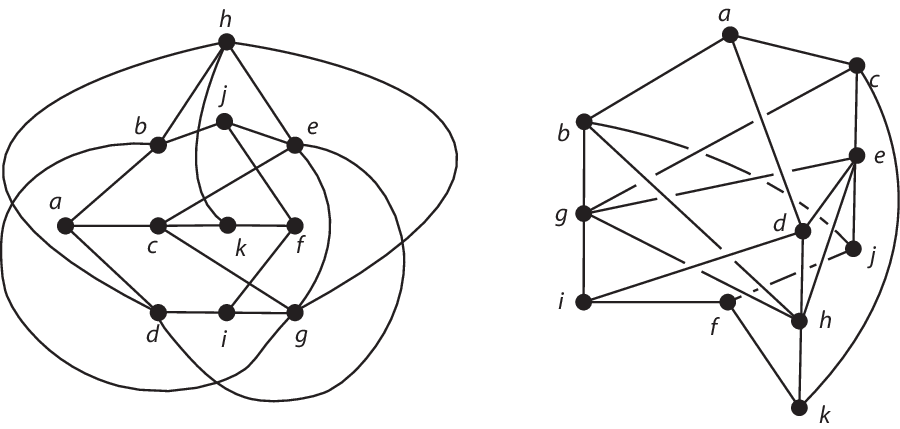}}$$
\caption{The graph $E_{11}$, and an embedding $\Gamma$ of $E_{11}$ with $\TSG(\Gamma) = \TSG_+(\Gamma) = \Z_3$.}
\label{F:E11}
\end{figure}

The vertices of degree 3 are $\{a, f, i, k, j\}$, the vertices of degree 4 are $\{b, c, d\}$ and the vertices of degree 5 are $\{e, g, h\}$.  Observe that $a$ is the only vertex of degree 3 adjacent to three vertices of degree 4, and $f$ is the only vertex of degree 3 adjacent to three other vertices of degree 3.  Hence $a$ and $f$ are both fixed by any automorphism of $E_{11}$.  Since $i, j, k$ are each adjacent to (in addition to $f$) one vertex of degree 4 and one of degree 5, the permutation of $\{i, j, k\}$ determines the permutations of $\{b, c, d\}$ and $\{e, g, h\}$.  So $\Aut(E_{11}) \cong D_3$, generated by $\alpha = (bcd)(geh)(ijk)$ and $\beta = (bd)(ge)(ij)$. The properties of the automorphisms are given in Table \ref{Ta:E11} (every nontrivial automorphism is conjugate to $\a$ or $\b$).

\begin{table}
\begin{tabular}{|l|c|c|c|c|c|}
\hline
\rule{0pt}{.15in}Automorphism & Order & Fixed subgraph & $\hookrightarrow S^1$ & $\hookrightarrow S^2$ & Path Lemma \\ \hline
\rule{0pt}{.15in}$\a = (bcd)(geh)(ijk)$ & 3 &  \begin{tikzpicture}[rotate=0] 
  \Vertex{a}
  \EA(a){b} 
\end{tikzpicture}  & Yes & Yes & N/A\textsuperscript{1} \\ \hline
\rule{0pt}{.15in}$\b = (bd)(ge)(ij)$ & 2 & \begin{tikzpicture}[rotate=0] 
     \Vertices{line}{a,b,c,d, e}
  \Edges(a,b,c, d)
    \NO(c){f}   
      \Edges(c,f)
    \AddVertexColor{white}{e}
  
\end{tikzpicture}   & No & Yes & $b-g-i-d$ \\ \hline
\end{tabular}
\caption{Automorphisms of $E_{11}$.}
\label{Ta:E11}
\end{table}

\begin{theorem} \label{T:E11}
The only nontrivial group that is positively realizable for $E_{11}$ is $\Z_3$.  $E_{11}$ is intrinsically chiral, so no other nontrivial group is realizable.
\end{theorem}
\begin{proof}
From Table \ref{Ta:E10}, $\beta$ is not postively realizable (by the Realizability Lemma) or negatively realizable (by the Path Lemma).  Moreover, $\a$ is not negatively realizable, since it has odd order, showing that $E_{11}$ is intrinsically chiral.

Figure \ref{F:E11} shows an embedding of $E_{11}$ which realizes the automorphism $\alpha$ by the rotation around the vertical axis through $a$ and $f$.  Since the only realizable automorphisms are the powers of $\a$, this shows $\Z_3$ is positively realizable, but no other nontrivial group is realizable.
\end{proof}

\subsection{The graph $C_{11}$}

The graph $C_{11}$ is obtained from $E_{10}$ (as shown in Figure \ref{F:E10}) by performing a $\nabla Y$-move on the triangle $ceg$, resulting in the graph shown in Figure~\ref{F:C11}.

\begin{figure} [htbp]
$$\scalebox{.8}{\includegraphics{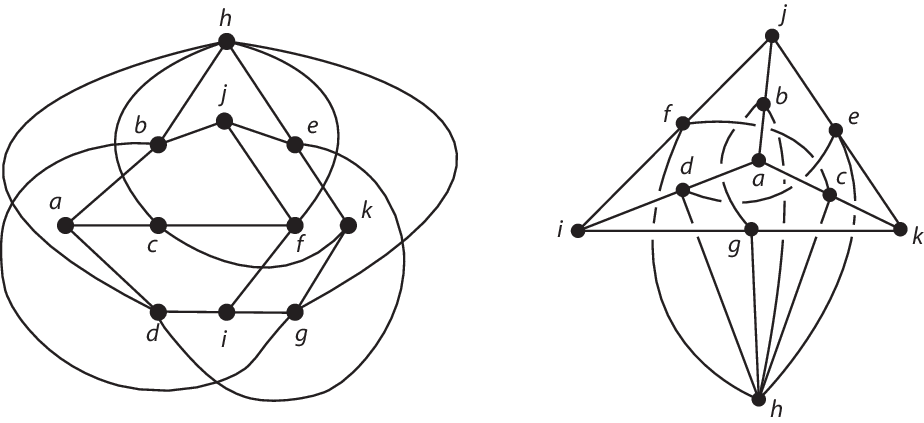}}$$
\caption{The graph $C_{11}$, and an embedding $\Gamma$ of $C_{11}$ with $\TSG(\Gamma) = \TSG_+(\Gamma) = \Z_3$.}
\label{F:C11}
\end{figure}

Since $h$ is the only vertex of degree 6, it is fixed by every automorphism of $C_{11}$.  Similarly, the sets $\{b, c, d, e, f, g\}$ of degree 4 vertices and $\{a, i, j, k\}$ of degree 3 vertices are fixed setwise.  Since each pair of degree 3 vertices have one degree 4 vertex adjacent to both, and it is a different common neighbor for each pair, the permutation of $\{a, i, j, k\}$ determines the permutation of $\{b, c, d, e, f, g\}$.  Since all the transpositions of $\{a, i, j, k\}$ give automorphisms (e.g. the transposition $(ai)$ induces the automorphism $(ai)(bf)(cg)$), we conclude that $\Aut(C_{11}) \cong S_4$, with generators $\a = (ai)(bf)(cg)$ and $\b = (ijk)(dbc)(feg)$, and with relations $\a^2 = \b^3 = (\a\b)^4 = 1$. The properties of the nontrivial automorphisms are given in Table \ref{Ta:C11} (listing one automorphism from each conjugacy class).

\begin{table}
\begin{tabular}{|l|c|c|c|c|c|}
\hline
\rule{0pt}{.15in}Automorphism & Order & fixed subgraph & $\hookrightarrow S^1$ & $\hookrightarrow S^2$ & Path Lemma \\ \hline
\rule{0pt}{.15in}$\a = (ai)(bf)(cg)$ & 2 &  \begin{tikzpicture}[rotate=0] 
  \Vertex{k}
  \NO(k){j} \EA(k){l} \NOEA(k){a} \NOWE(k){b}
  \Edges(j,k,l,j)   \Edges(j,a) \Edges(j,b)
\end{tikzpicture}  & No & Yes & $a-c-f-i$ \\ \hline
\rule{0pt}{.15in}$\b = (ijk)(dbc)(feg)$ & 3 & \begin{tikzpicture}[rotate=0] 
  \Vertex{a}
  \EA(a){b} 
\end{tikzpicture}  & Yes & Yes &N/A\textsuperscript{1} \\ \hline
\rule{0pt}{.15in}$\a\b = (aijk)(bg)(dfec)$ & 4 &   \begin{tikzpicture}[rotate=0] 
     \Vertices{line}{a,b}
    \AddVertexColor{white}{b}
\end{tikzpicture}  & Yes & Yes & N/A\textsuperscript{2} \\ \hline
\rule{0pt}{.15in}$(\a\b)^2 = (aj)(ik)(de)(fc)$ & 2 & \begin{tikzpicture}[rotate=0] 
   \Vertex{a}
  \EA(a){b} \NO(a){c}    \EA(b){d}    \EA(d){e}  
  \Edges(a,b,c,a)
     \AddVertexColor{white}{d,e}
\end{tikzpicture}  & No & Yes & $a-c-k-e-j$ \\ \hline
\end{tabular}
\caption{Automorphisms of $C_{11}.$}
\label{Ta:C11}
\end{table}

\begin{theorem} \label{T:C11}
The only nontrivial group that is positively realizable for $C_{11}$ is $\Z_3$.  Moreover, $C_{11}$ is intrinsically chiral, and no other nontrivial group is realizable.
\end{theorem}
\begin{proof}
From the Realizability Lemma and the Path Lemma, $\a$ and $(\a\b)^2$ are neither positively nor negatively realizable.  Since $(\a\b)^2$ is not realizable, neither is $\a\b$.  And $\b$ is not negatively realizable, since it has odd order, so $C_{11}$ is intrinsically chiral.

The only subgroups of $S_4$ which do not contain transpositions, products of two disjoint transpositions or 4-cycles are the subgroups isomorphic to $\Z_3$ induced by a 3-cycle, so the only possibly realizable group is $\Z_3$. Figure \ref{F:C11} shows an embedding of $C_{11}$ that positively realizes $\Z_3$, where the automorphism $\b$ is realized by the rotation around the vertical axis through $a$ and $h$.
\end{proof}

\subsection{The graph $C_{12}$}

The graph $C_{12}$ is obtained from $C_{11}$ (as shown in Figure \ref{F:C11}) by performing a $\nabla Y$-move on the triangle $cfh$, resulting in the graph shown in Figure~\ref{F:C12}.

\begin{figure} [htbp]
$$\scalebox{.8}{\includegraphics{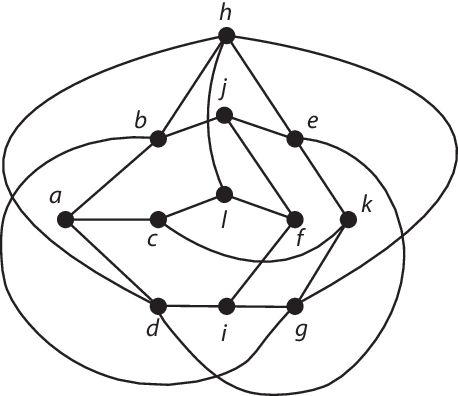}}$$
\caption{The graph $C_{12}$.}
\label{F:C12}
\end{figure}

Every automorphism of $C_{12}$ fixes $h$, the only vertex of degree 5.  Since $l$ is the only vertex of degree 3 adjacent to $h$, it is also fixed by every automorphism.  So the set $\{c, f\}$ (the other vertices adjacent to $l$) is fixed setwise.  The remaining vertices of the graph are contained in a unique 8-cycle $b-a-d-i-g-k-e-j-b$; the subgraph they induce is homeomorphic (in the graph theoretic sense) to a square, where $b, d, g, e$ are the vertices (connected by diagonals $\overline{bg}$ and $\overline{de}$), and $a, i, k, j$ are the midpoints of the sides.  This subgraph has an automorphism group isomorphic to $D_4$.  Since each automorphism of the square subgraph determines the permutation of $\{c, f\}$, the automorphism group of $C_{12}$ is also isomorphic to $D_4$, generated by $\alpha = (aikj)(bdge)(cf)$ and $\beta = (aj)(de)(ik)(cf)$. The properties of the nontrivial automorphisms are given in Table \ref{Ta:C12} (listing one automorphism from each conjugacy class).

\begin{table}
\begin{tabular}{|l|c|c|c|c|c|}
\hline
\rule{0pt}{.15in}Automorphism & Order & Fixed subgraph & $\hookrightarrow S^1$ & $\hookrightarrow S^2$ & Path Lemma \\ \hline
\rule{0pt}{.15in}$\a = (aikj)(bdge)(cf)$ & 4 &  \begin{tikzpicture}[rotate=0] 
  \Vertex{a}
  \EA(a){b}     \Edges(a,b)
\end{tikzpicture}  & Yes & Yes & $c-k-e-j-f$ \\ \hline
\rule{0pt}{.15in}$\a^2 = (ak)(ij)(bg)(de)$ & 2 &  \begin{tikzpicture}[rotate=0] 
  
    \Vertex{f}
  \EA(f){a} \WE(f){b} \SO(f){c}
   \Edges(c,b)
  \Edges(c,a)
    \Edges(c,f)
   
\end{tikzpicture}  & No & Yes & $a-d-i-g-k$ \\ \hline
\rule{0pt}{.15in}$\b = (aj)(de)(ik)(cf)$ & 2 & \begin{tikzpicture}[rotate=0] 
  \Vertex{k}
  \NO(k){j} \EA(k){l} \NOWE(k){b} \EA(l){m} 
  \Edges(j,k,l,j)   \Edges(j,a) \Edges(j,b)
  
      \AddVertexColor{white}{m}
\end{tikzpicture}   
  & No & Yes & $a-c-k-e-j$ \\ \hline
\rule{0pt}{.15in}$\a\b = (bd)(eg)(ij)$ & 2 & \begin{tikzpicture}[rotate=0] 
%
  \GraphInit[vstyle=Normal]
\Vertices{line}{A,B,C,D}
\NO(B){F} \NO(C){G}
  \Edges(A,B,C,D)   \Edges(B,F) \Edges(C,G)
\end{tikzpicture}  & No & Yes & $b-j-e-d$ \\ \hline
\end{tabular}
\caption{Automorphisms of $C_{12}.$}
\label{Ta:C12}
\end{table}

\begin{theorem} \label{T:C12}
No nontrivial group is positively realizable or realizable for $C_{12}$. Hence $C_{12}$ is intrinsically chiral.
\end{theorem}
\begin{proof}
From Table \ref{Ta:C12}, we easily see that $\a^2$, $\b$ and $\a\b$ are neither positively realizable (by the Realizability Lemma) or negatively realizable (by the Path Lemma).  Since $\a^2$ is not realizable, neither is $\a$.  So no nontrivial automorphism of $C_{12}$ is realizable.  Hence $C_{12}$ is intrinsically chiral, and no nontrivial group is realizable.
\end{proof}

\subsection{The graph $C_{13}$}

The graph $C_{13}$ is obtained from $C_{12}$ (as shown in Figure \ref{F:C12}) by performing a $\nabla Y$-move on the triangle $deh$, resulting in the graph shown in Figure~\ref{F:C13}.

\begin{figure} [htbp]
$$\scalebox{.8}{\includegraphics{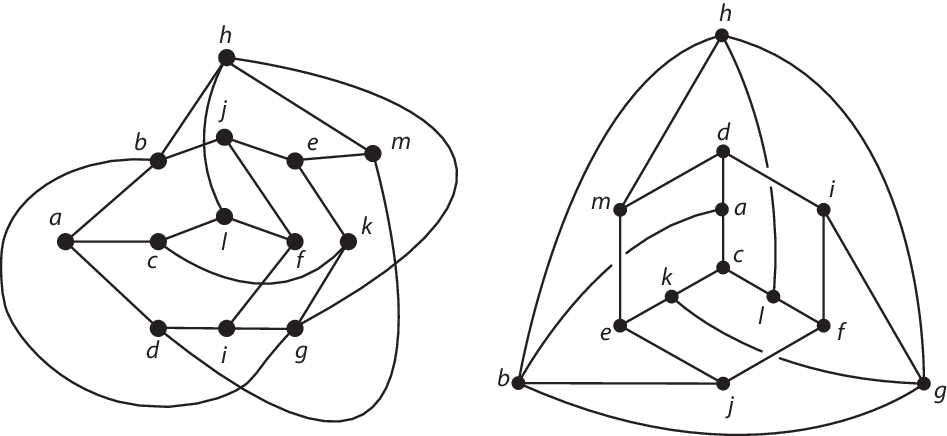}}$$
\caption{The graph $C_{13}$, and an embedding $\Gamma$ of $C_{13}$ with $\TSG(\Gamma) = \TSG_+(\Gamma) = \Z_3$.}
\label{F:C13}
\end{figure}

The graph $C_{13}$ has three vertices of degree 4, $\{b, g, h\}$, and the other ten have degree 3.  Of these ten, $a, i, j, k, l, m$ are adjacent to a vertex of degree 4, while $c, d, e, f$ are only adjacent to vertices of degree 3.  So these two sets of vertices are preserved setwise by every automorphism.  In fact, the permutation of $\{c, d, e, f\}$ determines the permutation of $\{a, i, j, k, l, m\}$ (each of which is adjacent to a different pair of vertices from $\{c, d, e, f\}$), and these determine the permutation of $\{b, g, h\}$.  Since every permutation of $\{c, d, e, f\}$ induces an automorphism on $C_{13}$, we conclude that $\Aut(C_{13}) \cong S_4$, with generators $\a = (cd)(il)(km)(gh)$ and $\b = (def)(imj)(akl)(bgh)$, and with relations $\a^2 = \b^3 = (\a\b)^4 = 1$. The properties of the nontrivial automorphisms are given in Table \ref{Ta:C13} (listing one automorphism from each conjugacy class).

\begin{table}
\begin{tabular}{|l|c|c|c|c|c|}
\hline
\rule{0pt}{.15in}Automorphism & Order & Fixed subgraph & $\hookrightarrow S^1$ & $\hookrightarrow S^2$ & Path Lemma \\ \hline
\rule{0pt}{.15in}$\a = (cd)(il)(km)(gh)$ & 2 & \begin{tikzpicture}[rotate=0] 
     \Vertices{line}{a,b,c,d, e}
  \Edges(a,b,c, d)
    \NO(c){f}   
      \Edges(c,f)
    \AddVertexColor{white}{e}
  
\end{tikzpicture} & No & Yes & $c-l-h-m-d$ \\ \hline
\rule{0pt}{.15in}$\b = (def)(imj)(akl)(bgh)$ & 3 & \begin{tikzpicture}[rotate=0] 
  \Vertex{d}
\end{tikzpicture}  & Yes & Yes & N/A\textsuperscript{1} \\ \hline
\rule{0pt}{.15in}$\a\b = (cdef)(amjl)(bh)$ & 4 & \begin{tikzpicture}[rotate=0] 
  \Vertex{a}
  \EA(a){b}    \EA(b){d}  \EA(d){e} 
  \Edges(a,b)
   \AddVertexColor{white}{e}
\end{tikzpicture} & Yes & Yes & $b-j-e-m-h$\\ \hline
\rule{0pt}{.15in}$(\a\b)^2 = (ce)(df)(aj)(ml)$ & 2 & \begin{tikzpicture}[rotate=0] 
  \Vertex{k}
  \NO(k){j} \EA(k){l} \NOEA(k){a} \NOWE(k){b}
  \Edges(j,k,l,j)   \Edges(j,a) \Edges(j,b)
\end{tikzpicture}    & No & Yes & $a-c-l-f-j$ \\ \hline
\end{tabular}
\caption{Automorphisms of $C_{13}.$}
\label{Ta:C13}
\end{table}

\begin{theorem} \label{T:C13}
The only nontrivial group that is positively realizable for $C_{13}$ is $\Z_3$.  No other nontrivial group is realizable; moreover, $C_{13}$ is intrinsically chiral.
\end{theorem}
\begin{proof}
From Table \ref{Ta:C13}, we see that $\a$ and $(\a\b)^2$ are not realizable (by the Realizability Lemma and the Path Lemma).  Hence $\a\b$ is also not realizable.  Since $\b$ has odd order, it is not negatively realizable, so $C_{13}$ is intrinsically chiral.

The only subgroups of $S_4$ which do not contain transpositions, products of two disjoint transpositions or 4-cycles are the subgroups isomorphic to $\Z_3$ induced by a 3-cycle, so the only possibly realizable group is $\Z_3$. Figure \ref{F:C13} shows an embedding of $C_{13}$ that positively realizes $\Z_3$, generated by the automorphism $\b$. 
\end{proof}

\subsection{The graph $N_{11}$}

The graph $N_{11}$ is obtained from $H_{12}$ (as shown in Figure \ref{F:H12}) by performing a $Y\nabla$-move removing vertex $h$ and adding triangle $jkl$, resulting in the graph shown in Figure~\ref{F:N11}.

\begin{figure} [htbp]
$$\scalebox{.8}{\includegraphics{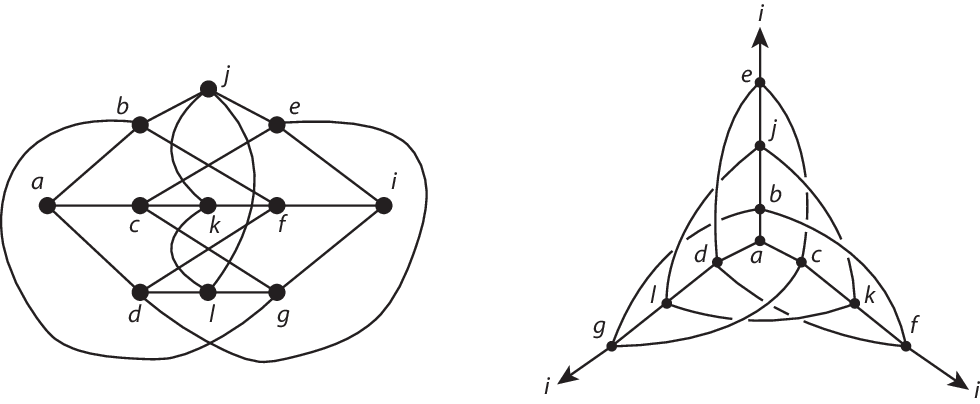}}$$
\caption{The graph $N_{11}$.}
\label{F:N11}
\end{figure}

The automorphisms of $N_{11}$ permute the set $\{a, i\}$ of degree 3 vertices, permute the set $\{j, k, l\}$ of vertices not adjacent to $a$ or $i$, and permute or interchange the sets $\{b, c, d\}$ and $\{e, f, g\}$.  The automorphisms are completely determined by the permutation of $\{j, k, l\}$ and the permutation of $\{a, i\}$.  These permutations commute, so $\Aut(N_{11}) \cong D_3 \times \Z_2$. The properties of the automorphisms are given in Table \ref{Ta:N11}.

\begin{table}
\begin{tabular}{|l|c|c|c|c|c|}
\hline
\rule{0pt}{.15in}Automorphism & Order & Fixed subgraph & $\hookrightarrow S^1$ & $\hookrightarrow S^2$ & Path Lemma \\ \hline
\rule{0pt}{.15in}$\a = (bcd)(efg)(jkl)$ & 3 & \begin{tikzpicture}[rotate=0] 
  \Vertex{a}
  \EA(a){b} 
\end{tikzpicture}  & Yes & Yes & N/A\textsuperscript{1} \\ \hline
\rule{0pt}{.15in}$\b = (bc)(ef)(jk)$ & 2 & \begin{tikzpicture}[rotate=0] 
  \Vertices{line}{a,d,g,i,l,x}
  \Edges(a,d,g,i,l)
     \AddVertexColor{white}{x}
\end{tikzpicture}  & Yes & Yes & $j-b-f-k$ \\ \hline
\rule{0pt}{.15in}$\g = (ai)(be)(cf)(dg)$ & 2 & \begin{tikzpicture}[rotate=0] 
  \Vertex{k}
  \NO(k){j} \EA(k){l}
  \Edges(j,k,l,j)
\end{tikzpicture}  & Yes & Yes & $a-b-f-i$ \\ \hline
\rule{0pt}{.15in}$\a\g = (ai)(bfdecg)(jkl)$ & 6 & $\emptyset$ & Yes & Yes & N/A\textsuperscript{2} \\ \hline
\rule{0pt}{.15in}$\b\g = (ai)(bf)(ec)(dg)(jk)$ & 2 & \begin{tikzpicture}[rotate=0] 
  \Vertices{line}{h,x,y,z}
  \AddVertexColor{white}{x,y,z}
\end{tikzpicture}  & Yes & Yes & $a-c-g-i$ \\ \hline
\end{tabular}
\caption{Automorphisms of $N_{11}$.}
\label{Ta:N11}
\end{table}

\begin{theorem}\label{T:N11}
The only groups that are positively realizable for $N_{11}$ are $D_3 \times \Z_2$ and its subgroups $\Z_6, D_3, \Z_3, D_2$ and $\Z_2$.  Moreover, $N_{11}$ is intrinsically chiral, and no other nontrivial group is realizable.
\end{theorem}
\begin{proof}
From Table \ref{Ta:N11} we see that $\a$ is not negatively realizable since it has odd order, and $\b$, $\g$ and $\b\g$ are not negatively realizable by the Path Lemma.  Since $(\a\g)^3 = \g$, the automorphism $\a\g$ is also not negatively realizable.  So none of the automorphisms are negatively realizable, and $N_{11}$ is intrinsically chiral.

The embedding of $N_{11}$ on the right in Figure \ref{F:N11} positively realizes $D_3 \times \Z_2$.  In this embedding, $i$ is embedded at the point in $S^3$ antipodal to $a$, and they are interchanged by the order 2 rotation around the triangle $jkl$, which realizes $\g$. Then $\a$ is realized by rotation around vertex $a$, and $\b$ by rotation about the axis through vertices $d, g, l$. Since the edge $\overline{de}$ is not fixed by any automorphism, the subgroups of $D_3 \times \Z_2$ are also positively realizable by the Subgroup Theorem.  Since $D_3 \times \Z_2$ is the full automorphism group, there are no other realizable groups.
\end{proof}

\subsection{The graph $N_{10}$}

The graph $N_{10}$ is obtained from $N_{11}$ (as shown in Figure \ref{F:N11}) by performing a $Y\nabla$-move removing vertex $i$ and adding triangle $efg$, resulting in the graph shown in Figure~\ref{F:N10}.

\begin{figure} [htbp]
$$\scalebox{.8}{\includegraphics{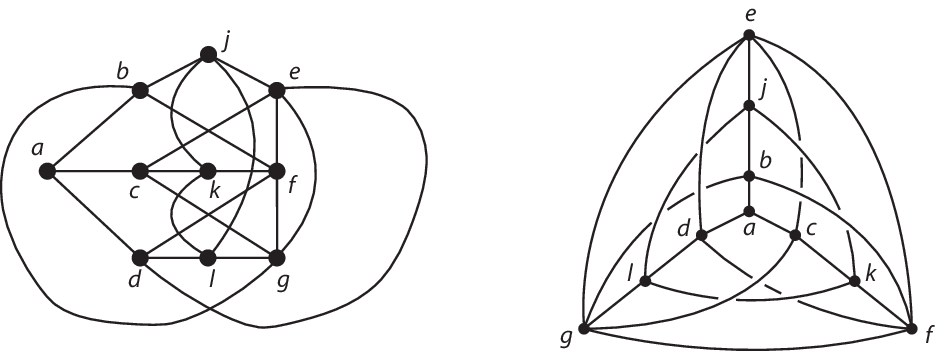}}$$
\caption{The graph $N_{10}$, and an embedding $\Gamma$ of $N_{10}$ with $\TSG(\Gamma) = \TSG_+(\Gamma) = D_3$.}
\label{F:N10}
\end{figure}

Every automorphism of $N_{10}$ fixes vertex $a$ (the only vertex of degree 3), the set $\{b, c, d\}$ (the vertices adjacent to $a$), the set $\{j, k, l\}$ (the other vertices of degree 4) and the set $\{e, f, g\}$ (the vertices of degree 5).  The permutation of $\{j, k, l\}$ determines the permutation of the other vertices, so $\Aut(N_{10}) \cong D_3$. The properties of the automorphisms are given in Table \ref{Ta:N10}.

\begin{table}
\begin{tabular}{|l|c|c|c|c|c|}
\hline
\rule{0pt}{.15in}Automorphism & Order & Fixed subgraph & $\hookrightarrow S^1$ & $\hookrightarrow S^2$ & Path Lemma \\ \hline
\rule{0pt}{.15in}$\a = (bcd)(efg)(jkl)$ & 3 & \begin{tikzpicture}[rotate=0] 
  \Vertex{a}
\end{tikzpicture}   & Yes & Yes & N/A\textsuperscript{1} \\ \hline
\rule{0pt}{.15in}$\b = (bc)(ef)(jk)$ & 2 &  \begin{tikzpicture}[rotate=0] 
  \Vertices{line}{a,d,g,l,x,y}
  \Edges(a,d,g,l)
     \AddVertexColor{white}{x,y}
\end{tikzpicture}& Yes & Yes & $b-f-k-c$ \\ \hline
\end{tabular}
\caption{Automorphisms of $N_{10}$.}
\label{Ta:N10}
\end{table}

\begin{theorem}\label{T:N10}
The only groups that are positively realizable for $N_{10}$ are $D_3$ and its subgroups $\Z_3$ and $\Z_2$.  $N_{10}$ is intrinsically chiral, and no other nontrivial group is realizable.
\end{theorem}
\begin{proof}
From Table \ref{Ta:N10}, we see that $\a$ is not negatively realizable because it has odd order, and $\b$ is not negatively realizable by the Path Lemma.  Hence $N_{10}$ is intrinsically chiral. The embedding of $N_{10}$ on the right in Figure \ref{F:N10} positively realizes $D_3$.  Since the edge $\overline{de}$ is not fixed by any automorphism, the subgroups of $D_3$ are also positively realizable by the Subgroup Theorem.  Since $D_3$ is the full automorphism group, there are no other realizable groups.
\end{proof}

\subsection{The graph $N_{9}$}

The graph $N_{9}$ is obtained from $N_{10}$ (as shown in Figure \ref{F:N10}) by performing a $Y\nabla$-move removing vertex $a$ and adding triangle $bcd$, resulting in the graph shown in Figure~\ref{F:N10}.

\begin{figure} [htbp]
$$\scalebox{.8}{\includegraphics{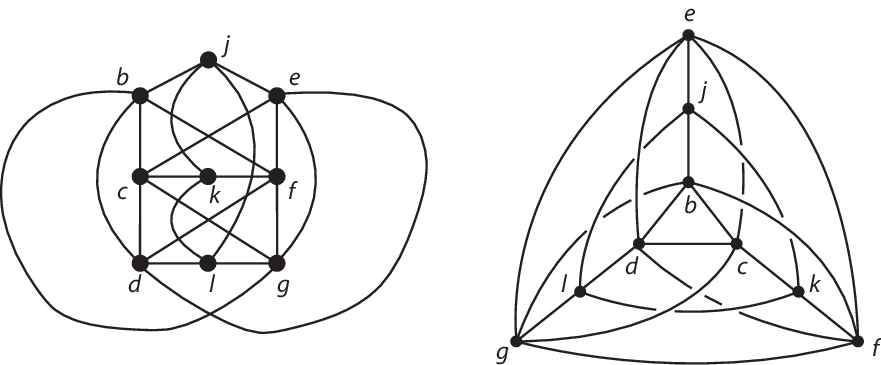}}$$
\caption{The graph $N_{9}$, and an embedding $\Gamma$ of $N_{9}$ with $\TSG(\Gamma) = \TSG_+(\Gamma) = D_3 \times \Z_2$.}
\label{F:N9}
\end{figure}

Every automorphism of $N_9$ preserves the set $\{j,k,l\}$ of degree 4 vertices and the set $\{b,c,d,e,f,g\}$ of degree 5 vertices.  Morever, the permutation of $\{j, k, l\}$ permutes the three pairs $\{b,e\}$, $\{c,f\}$, $\{d,g\}$; the vertices in each pair can also be interchanged.  So the automorphism group is generated by the permutations $\alpha = (jkl)(bcd)(efg)$, $\beta = (cd)(lk)(gf)$, $\gamma_j = (be)$, $\gamma_k = (cf)$ and $\gamma_l = (dg)$.  Automorphisms $\alpha$ and $\beta$ generate a subgroup isomorphic to $D_3$, while $\gamma_j, \gamma_k,$ and $\gamma_l$ generate $(\Z_2)^3$.  The automorphism group is a semidirect product of these two subgroups, with the homeomorphism $\phi: D_3 \rightarrow \Aut((\Z_2)^3)$ defined by $\phi(g)(\gamma_x) = g\gamma_x g^{-1} = \gamma_{g(x)}$. So $\Aut(N_9) \cong (\Z_2)^3\semi D_3$. The properties of the automorphisms are given in Table \ref{Ta:N9} (we list one automorphism from each conjugacy class, as in Table \ref{Ta:F10}).

\begin{table}
\begin{tabular}{|l|c|c|c|c|c|}
\hline
\rule{0pt}{.15in}Automorphism & Order & Fixed subgraph & $\hookrightarrow S^1$ & $\hookrightarrow S^2$ & Path Lemma \\ \hline
\rule{0pt}{.15in}$\a = (jkl)(bcd)(efg)$ & 3 & $\emptyset$ & Yes & Yes & N/A\textsuperscript{1} \\ \hline
\rule{0pt}{.15in}$\a\g_j = (jkl)(bfgecd)$ & 6 & $\emptyset$ & Yes & Yes & N/A\textsuperscript{1} \\ \hline
\rule{0pt}{.15in}$\b = (cd)(lk)(gf)$ & 2 &  \begin{tikzpicture}[rotate=0] 
  \Vertices{line}{b,e,j,x,y,z}
  \Edges(b,e,j)
     \AddVertexColor{white}{x,y,z}
\end{tikzpicture} & Yes & Yes & $c-g-l-d$ \\ \hline
\rule{0pt}{.15in}$\b\g_k = (cgfd)(lk)$ & 4 & \begin{tikzpicture}[rotate=0] 
  \Vertices{line}{b,e,j,x}
  \Edges(b,e,j)
     \AddVertexColor{white}{x}
\end{tikzpicture}  & Yes & Yes & $l-d-c-k$ \\ \hline
\rule{0pt}{.15in}$\b\g_j = (be)(cd)(lk)(gf)$ & 2 &  \begin{tikzpicture}[rotate=0] 
  \Vertices{line}{j,x,y,z}
     \AddVertexColor{white}{x,y,z}
\end{tikzpicture}   & Yes & Yes & $c-g-l-d$ \\ \hline
\rule{0pt}{.15in}$\b\g_j\g_k = (be)(cgfd)(lk)$ & 4 & \begin{tikzpicture}[rotate=0] 
  \Vertices{line}{j,x}
     \AddVertexColor{white}{x}
\end{tikzpicture} & Yes & Yes & N/A\textsuperscript{2} \\ \hline
\rule{0pt}{.15in}$\g_j = (be)$ & 2 & 
\begin{tikzpicture}[rotate=90] 
  \Vertex{f}
  \WE[unit=1](f){i} \NO(f){e} \SO(f){g} \NO(i){h} \SO(i){d} \SOWE(g){l}
  \Edges(e,i,g,h,f,d,e)
  \Edges(e,h)  \Edges(i,f) \Edges(d,g)
  
  \tikzset{EdgeStyle/.style = {-
,bend left=75}}
\Edge(f)(l)
  
    \tikzset{EdgeStyle/.style = {-
,bend right=75}}
\Edge(i)(l)
    
\end{tikzpicture} 
 & No & No & N/A\textsuperscript{3} \\ \hline
\rule{0pt}{.15in}$\g_j\g_k = (be)(cf)$ & 2 & \begin{tikzpicture}[rotate=0] 
  \Vertex{k}
  \NO(k){j} \EA(k){l} \NOEA(k){a} \NOWE(k){b}
  \Edges(j,k,l,j)   \Edges(j,a) \Edges(j,b)
\end{tikzpicture}    & No & Yes & $b-c-e$ \\ \hline
\rule{0pt}{.15in}$\g_j\g_k\g_l = (be)(cf)(dg)$ & 2 &  \begin{tikzpicture}[rotate=0] 
  \Vertex{k}
  \NO(k){j} \EA(k){l}
  \Edges(j,k,l,j)
\end{tikzpicture}   & Yes & Yes & $b-c-e$ \\ \hline
\end{tabular}
\caption{Automorphisms of $N_{9}$.}
\label{Ta:N9}
\end{table}

\begin{theorem}\label{T:N9}
The only groups that are positively realizable for $N_9$ are $D_3 \times \Z_2$ and its subgroups $\Z_6, D_3, \Z_3, D_2$ and $\Z_2$.  Moreover, $N_9$ is intrinsically chiral, and no other nontrivial group is realizable.
\end{theorem}
\begin{proof}
From Table \ref{Ta:N9}, we see that $\g_j$ and $\g_j\g_k$ are not realizable (by the Realizability Lemma and the Path Lemma), $\b$, $\b\g_k$, $\b\g_j$ and $\g_j\g_k\g_l$ are not negatively realizable by the Path Lemma, and $\a$ is not negatively realizable because it has odd order.  Since $(\a\g_j)^3 = \g_j\g_k\g_l$, it is also not negatively realizable.  Finally, $(\b\g_k)^2 = (\b\g_j\g_k)^2 = \g_k\g_l$, which is in the conjugacy class of $\g_j\g_k$, so it is not realizable; therefore $\b\g_k$ and $\b\g_j\g_k$ are also not realizable.  Hence no automorphisms are negatively realizable, so $N_9$ is intrinsically chiral.

The only positively realizable automorphisms are those conjugate to $\a$, $\a\g_j$, $\b$, $\b\g_j$, and $\g_j\g_k\g_l$. The embedding shown on the right in Figure \ref{F:N9} positively realizes $D_3 \times \Z_2$, generated by $\a$ (the rotation of order 3 around the center), $\b$ (the rotation of order 2 around the axis through vertices $e, j, b$), and $\g_j\g_k\g_l$ (the rotation of order 2 around the triangle $jkl$).  Including any other automorphisms generates a group that contains non-realizable automorphisms, so the topological symmetry group is $D_3 \times \Z_2$. Since the edge $\overline{de}$ is not fixed by any element of the group, every subgroup of $D_3 \times \Z_2$ is positively realizable by the Subgroup Theorem. Using {\em Sage} \cite{sage}, we verified that any subgroup of $\Aut(N_9)$ not isomorphic to a subgroup of $D_3\times \Z_2$ must contain non-realizable automorphisms, so these are the only realizable topological symmetry groups.
\end{proof}

\subsection{The graph $N'_{12}$}

The graph $N'_{12}$ is obtained from $C_{13}$ (as shown in Figure \ref{F:C13}) by performing a $Y\nabla$-move removing vertex $f$ and adding triangle $ijl$, resulting in the graph shown in Figure~\ref{F:N12prime}.

\begin{figure} [htbp]
$$\scalebox{.8}{\includegraphics{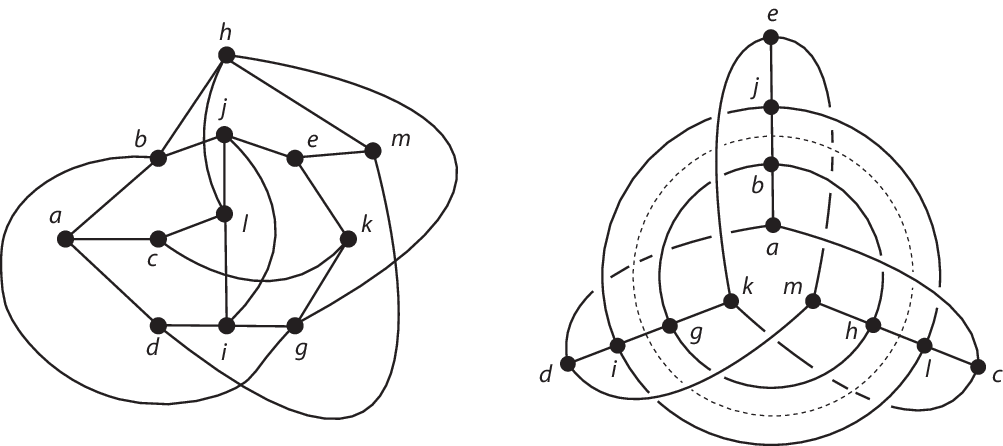}}$$
\caption{The graph $N'_{12}$, and an embedding $\Gamma$ of $N'_{12}$ with $\TSG(\Gamma) = \TSG_+(\Gamma) = D_6$.}
\label{F:N12prime}
\end{figure}

The graph $N'_{12}$ has six vertices of degree 3.  These form a 6-cycle in the graph: $a-c-k-e-m-d-a$.  This 6-cycle must be preserved by any automorphism of the graph, so the automorphism group of the graph is a subgroup of $D_6$.  Since every automorphism of the 6-cycle can be realized by an automorphism of the graph, we conclude $\Aut(N'_{12}) \cong D_6$, generated by $\alpha = (ackemd)(blgjhi)$ and $\beta = (cd)(li)(hg)(mk)$. The properties of the automorphisms are given in Table \ref{Ta:N12prime}.

\begin{table}
\begin{tabular}{|l|c|c|c|c|c|}
\hline
\rule{0pt}{.15in}Automorphism & Order & Fixed subgraph & $\hookrightarrow S^1$ & $\hookrightarrow S^2$ & Path Lemma \\ \hline
\rule{0pt}{.15in}$\a = (ackemd)(blgjhi)$ & 6 & $\emptyset$ & Yes & Yes & N/A\textsuperscript{1} \\ \hline
\rule{0pt}{.15in}$\a^2 = (akm)(ced)(bgh)(lji)$ & 3 & $\emptyset$ & Yes & Yes & N/A\textsuperscript{1} \\ \hline
\rule{0pt}{.15in}$\a^3 = (ae)(cm)(kd)(bj)(lh)(gi)$ & 2 & \begin{tikzpicture}[rotate=0] 
  \Vertices{line}{x,y,z}
     \AddVertexColor{white}{x,y,z}
\end{tikzpicture} & Yes & Yes & $a-c-k-e$ \\ \hline
\rule{0pt}{.15in}$\b = (cd)(li)(hg)(mk)$ & 2 & \begin{tikzpicture}[rotate=0] 
  \Vertices{line}{a,b,e,j}
  \Edges(a,b,e,j)
\end{tikzpicture} & Yes & Yes & $c-k-g-i-d$ \\ \hline
\rule{0pt}{.15in}$\a\b = (ac)(dk)(em)(bl)(gi)(hj)$ & 2 &  \begin{tikzpicture}[rotate=0] 
  \Vertices{line}{x,y,z}
     \AddVertexColor{white}{x,y,z}
\end{tikzpicture}  & Yes & Yes & $d-a-b-g-k$ \\ \hline
\end{tabular}
\caption{Automorphisms of $N'_{12}.$}
\label{Ta:N12prime}
\end{table}

\begin{theorem}\label{T:N12prime}
The only groups that are positively realizable for $N'_{12}$ are $D_6$ and its subgroups $\Z_6$, $D_3$, $\Z_3$, $D_2$ and $\Z_2$.  $N'_{12}$ is intrinsically chiral, and no other nontrivial group is realizable.
\end{theorem}
\begin{proof}
From Table $\ref{Ta:N12prime}$, we can see that $\a^3$, $\b$ and $\a\b$ are not negatively realizable by the Path Lemma, and $\a^2$ is not negatively realizable because it has odd order. Finally, since $\a^3$ is not negatively realizable, neither is $\a$, so $N'_{12}$ is intrinsically chiral. 

The embedding on the right in Figure \ref{F:N12prime} positively realizes $\Aut(N'_{12}) \cong D_6$: $\alpha$ is realized by the glide rotation consisting of an order 3 rotation around the center of the figure, and an order 2 rotation about the dashed circle in the diagram (going through the midpoints of $\overline{bj}$, $\overline{hl}$ and $\overline{gi}$); $\beta$ is realized by the order 2 rotation about the vertical axis through $a, b, j, e$.

Since the edge $\overline{bh}$ is not fixed pointwise by any automorphism of the graph, we can also positively realize every subgroup of $D_6$, by the Subgroup Theorem.
\end{proof}

\subsection{The graph $N'_{11}$}

The graph $N'_{11}$ is obtained from $N'_{12}$ (as shown in Figure \ref{F:N12prime}) by performing a $Y\nabla$-move removing vertex $m$ and adding triangle $deh$, resulting in the graph shown in Figure~\ref{F:N11prime}.

\begin{figure} [htbp]
$$\scalebox{.8}{\includegraphics{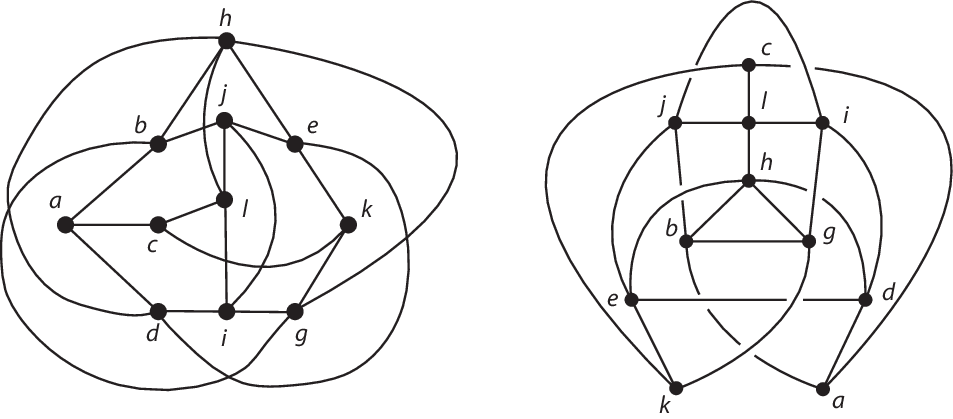}}$$
\caption{The graph $N'_{11}$, and an embedding $\Gamma$ of $N'_{11}$ with $\TSG(\Gamma) = \TSG_+(\Gamma) = \Z_2$.}
\label{F:N11prime}
\end{figure}

Every automorphism of $N'_{11}$ fixes $h$ (the only vertex of degree 5), $c$ (the only vertex of degree 3 with two neighbors of degree 3) and $l$ (the only vertex adjacent to both $h$ and $c$).  Moreover, every automorphism fixes (setwise) the sets $\{a, k\}$ (the other vertices of degree 3), $\{j, i\}$ (the other vertices adjacent to $l$) and $\{b, d, e, g\}$ (the remaining vertices).  The permutations of $\{a,k\}$ and $\{j,i\}$ determine the permutation of the remaining vertices, so $\Aut(N'_{11}) \cong \Z_2 \times \Z_2 = D_2$, generated by $\alpha = (ak)(be)(dg)$ and $\beta = (ji)(bd)(eg)$. The properties of the automorphisms are given in Table \ref{Ta:N11prime}.

\begin{table}
\begin{tabular}{|l|c|c|c|c|c|}
\hline
\rule{0pt}{.15in}Automorphism & Order & Fixed subgraph & $\hookrightarrow S^1$ & $\hookrightarrow S^2$ & Path Lemma \\ \hline
\rule{0pt}{.15in}$\a = (ak)(be)(dg)$ & 2 &   \begin{tikzpicture}[rotate=0] 
  \Vertex{k}
  \NO(k){j} \EA(k){l} \NOEA(k){a} \NOWE(k){b}
  \Edges(j,k,l,j)   \Edges(j,a) \Edges(j,b)
\end{tikzpicture}  & No & Yes & $a-d-e-k$ \\ \hline
\rule{0pt}{.15in}$\b = (ji)(bd)(eg)$ & 2 & \begin{tikzpicture}[rotate=0] 
     \Vertices{line}{a,b,c,d}
  \Edges(a,b,c, d)
    \NO(c){f}   
      \Edges(c,f)

\end{tikzpicture}  & No & Yes & $b-g-i-d$ \\ \hline
\rule{0pt}{.15in}$\a\b = (ak)(ji)(bg)(de)$ & 2 &  \begin{tikzpicture}[rotate=0] 
  \Vertices{line}{b,e,j,x,y,z}
  \Edges(b,e,j)
     \AddVertexColor{white}{x,y,z}
\end{tikzpicture}  & Yes & Yes & $a-d-i-g-k$ \\ \hline
\end{tabular}
\caption{Automorphisms of $N'_{11}.$}
\label{Ta:N11prime}
\end{table}

\begin{theorem}\label{T:N11prime}
The only group that is positively realizable for $N'_{11}$ is $\Z_2$.  $N'_{11}$ is intrinsically chiral, and no other nontrivial group is realizable.
\end{theorem}
\begin{proof}
From Table \ref{Ta:N11prime}, none of the automorphisms are negatively realizable, by the Path Lemma, so $N'_{11}$ is intrinsically chiral.  Also, by the Realizability Lemma, neither $\a$ nor $\b$ is positively realizable.  However, the product $\alpha\beta = (ak)(ji)(bg)(ed)$ is positively realizable, as shown on the right in Figure \ref{F:N11prime}.  So the group $\Z_2$ is positively realizable, and no other nontrivial groups are realizable.
\end{proof}

\subsection{The graph $N'_{10}$}

The graph $N'_{10}$ is obtained from $N'_{11}$ (as shown in Figure \ref{F:N11prime}) by performing a $Y\nabla$-move removing vertex $c$ and adding triangle $akl$, resulting in the graph shown in Figure~\ref{F:N10prime}.

\begin{figure} [htbp]
$$\scalebox{.8}{\includegraphics{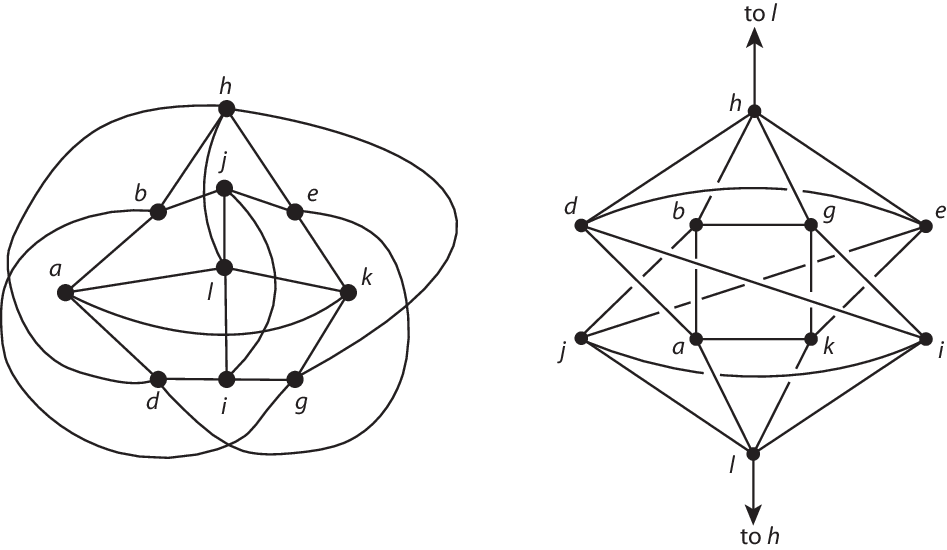}}$$
\caption{The graph $N'_{10}$, and an embedding $\Gamma$ of $N'_{10}$ with $\TSG(\Gamma) = \TSG_+(\Gamma) = D_2$.}
\label{F:N10prime}
\end{figure}

The graph $N'_{10}$ has two vertices of degree 5, $\{h, l\}$, and eight vertices of degree 4.  The subgraph induced by the vertices of degree 4 is isomorphic to a M\"{o}bius ladder $M_4$, which has automorphism group $D_8$ \cite{fl}.  These automorphisms extend to the rest of $N'_{10}$, so $\Aut(N'_{10}) \cong D_8$, generated by $\a = (hl)(abjekgid)$ and $\b = (hl)(ab)(dj)(ei)(gk)$. The properties of the automorphisms are given in Table \ref{Ta:N10prime}.

\begin{table}
\begin{tabular}{|l|c|c|c|c|c|}
\hline
\rule{0pt}{.15in}Automorphism & Order & Fixed subgraph & $\hookrightarrow S^1$ & $\hookrightarrow S^2$ & Path Lemma \\ \hline
\rule{0pt}{.15in}$\a = (hl)(abjekgid)$ & 8 &  \begin{tikzpicture}[rotate=0] 
  \Vertices{line}{x}
     \AddVertexColor{white}{x}
\end{tikzpicture}  & Yes & Yes & N/A\textsuperscript{2}\\ \hline
\rule{0pt}{.15in}$\a^2 = (ajki)(begd)$ & 4 &  \begin{tikzpicture}[rotate=0] 
  \Vertices{line}{h,l}
\Edges(h,l)
\end{tikzpicture}  & Yes & Yes & N/A\textsuperscript{1}  \\ \hline
\rule{0pt}{.15in}$\a^3 = (hl)(aeibkdjg)$ & 8 & \begin{tikzpicture}[rotate=0] 
  \Vertices{line}{x}
     \AddVertexColor{white}{x}
\end{tikzpicture}  & Yes & Yes & N/A\textsuperscript{2} \\ \hline
\rule{0pt}{.15in}$\a^4 = (ak)(bg)(ji)(ed)$ & 2 &  \begin{tikzpicture}[rotate=0] 
  \Vertices{line}{x,y,z}
     \AddVertexColor{white}{x,y,z}
\end{tikzpicture}  & Yes & Yes & $a-d-i-g-k$ \\ \hline
\rule{0pt}{.15in}$\b = (hl)(ab)(dj)(ei)(gk)$ & 2 &  \begin{tikzpicture}[rotate=0] 
  \Vertices{line}{x,y,z}
     \AddVertexColor{white}{x,y,z}
\end{tikzpicture}  & Yes & Yes & $d-i-l-j$ \\ \hline
\rule{0pt}{.15in}$\a\b = (aj)(de)(ik)$ & 2 & \begin{tikzpicture}[rotate=0] 
  \Vertex{f}
  \EA(f){i} \WE(f){h} \NO(f){e}  \EA(i){j}
  \Edges(f,i,e,f)
    \Edges(f,h)
         \AddVertexColor{white}{j}
\end{tikzpicture}   & No & Yes & $a-d-i-j$ \\ \hline
\end{tabular}
\caption{Automorphisms of $N'_{10}.$}
\label{Ta:N10prime}
\end{table}

\begin{theorem}\label{TN10prime}
The only groups that are positively realizable for $N'_{10}$ are $D_2$ and $\Z_2$.  Moreover, $N'_{10}$ is intrinsically chiral, and no other nontrivial groups are realizable.
\end{theorem}
\begin{proof}
From Table \ref{Ta:N10prime}, we see that $\a^4$, $\b$ and $\a\b$ are not negatively realizable by the Path Lemma, and $\a\b$ is also not positively realizable. Since $\a^2$ fixes more than two points ($h$, $l$ and the edge between them), and does not have order 2, it is not negatively realizable by the Realizability Lemma.

Towards contradiction, suppose the automorphism $\a^2 = (ajki)(begd)$ is positively realizable for an embedding $\Gamma$ of $N'_{10}$.  Since $\a^2$ fixes $h$ and $l$, it must be realized by a rotation; hence $\a^4$ is also a rotation.  Since $\a^4$ interchanges $a$ and $k$, it must fix a point on the edge $\overline{ak}$, so $\overline{ak}$ crosses the axis of rotation.  But then $\a^2$ would also fix a point on $\overline{ak}$, which contradicts the fact that $\a^2$ sends $\overline{ak}$ to the disjoint edge $\overline{ji}$.  Hence $\a^2$ is not positively realizable (and so not realizable).  Since $\a^2$ is a power of each of $\a, \a^3, \a^5, \a^6$ and $\a^7$, none of these are realizable. So the only power of $\a$ which may be realizable is $\a^4$. Since none of the automorphisms are negatively realizable, $N'_{10}$ is intrinsically chiral.

The only automorphisms which could be realizable are $\a^4$ and $\b$ (and their conjugates). Observe that $\a^4$ is only conjugate to itself, and multiplying $\b$ by any conjugate other than $\a^4\b$ gives a power of $\a$ which is not realizable. So the largest possible positively realizable group is isomorphic to $D_2$, with $\a^4$, $\b$ and $\a^4\b$.  The embedding shown on the right in Figure \ref{F:N10prime} positively realizes $D_2$ by realizing $\a^4$ by an order 2 rotation around the axis through $h$ and $l$, and realizing $\b$ by an order two rotation about the horizontal line through the middle of the diagram.  Since the edge $\overline{hd}$ is not fixed by any of these automorphisms, we can also positively realize $\Z_2$ by the Subgroup Theorem. Hence the only groups that are realizable for $N'_{10}$ are $D_2$ and $\Z_2$.
\end{proof}

\section{Conclusions and Questions}\label{S:conclusion}

Aside from adding to our library of graphs for which we know the topological symmetry groups, the purpose of this project was to explore whether there are any relationships among the topological symmetry groups of graphs related by $\nabla Y$ and $Y \nabla$ moves. The most interesting observation is that none of the graphs in the Heawood family have any topological symmetry groups that are realizable but not positively realizable; in fact, every graph in the Heawood family is intrinsically chiral!  This is in contrast to the Petersen family (another family related by $\nabla Y$ moves, shown in Figure \ref{F:petersen}), where most of the graphs do have topological symmetry groups that are realizable but not positively realizable \cite{efmsw}.

\begin{figure} [htbp]
$$\scalebox{.8}{\includegraphics{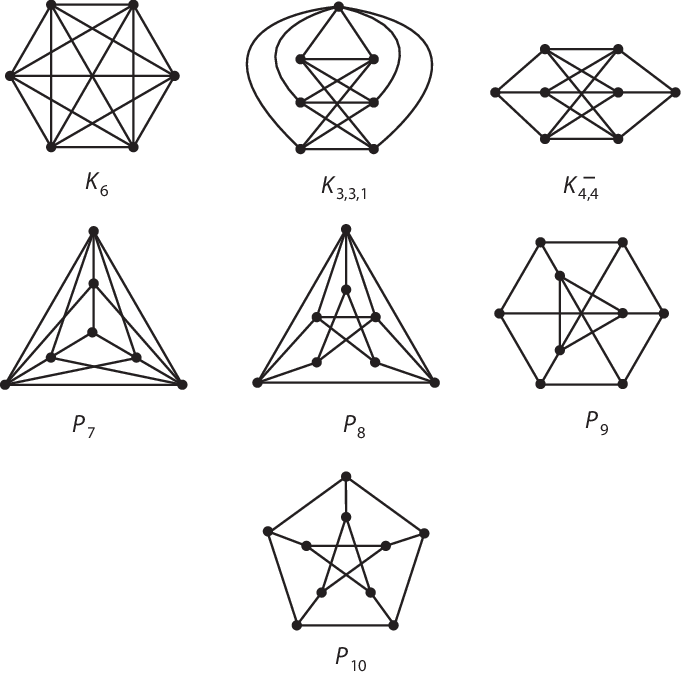}}$$
\caption{The Petersen family of graphs.}
\label{F:petersen}
\end{figure}

Putting together the results of all our theorems in Sections \ref{S:background} and \ref{S:details}, we have proved the following theorem.

\begin{theorem}
The graphs of the Heawood family are all intrinsically chiral.
\end{theorem}

Intrinsically chiral graphs are of particular interest in both chemistry and topology, and substantial work has been done on proving particular graphs are intrinsically chiral \cite{ckn1, ckn2, ffn, fw}. These proofs are often difficult and performed only for small graphs or very special families. In light of our results for the Heawood family, it is natural to ask whether $\nabla Y$ and $Y\nabla$ moves always preserve intrinsic chirality.  In general, however, they do not, as we can see in the following examples:
\begin{itemize}
	\item Let $\g_1$ be the wedge of two copies of $K_7$ joined at a single vertex $v$, and $\Gamma_1$ be an embedding where one copy of $K_7$ is embedded as the mirror image of the other.  Then $\g_1$ is not intrinsically chiral.  However, if we apply a $\nabla Y$ move to one of the copies of $K_7$, then the resulting graph $\g_2$ will be intrinsically chiral (since $K_7$ is intrinsically chiral \cite{ffn}).  A similar construction can be done with the wedge of two copies of $C_{14}$ using a $Y \nabla$ move ($C_{14}$ is also shown to be intrinsically chiral in \cite{ffn}).  So neither $\nabla Y$ nor $Y\nabla$ moves preserve intrinsic chirality in general.
	\item We can also note that, in the Petersen family of graphs, the graph $P_8$ is intrinsically chiral from the results in \cite{efmsw}, but neither $P_7$ (obtained from $P_8$ by a $Y\nabla$ move) nor $P_9$ (obtained from $P_8$ by a $\nabla Y$ move) are intrinsically chiral.
\end{itemize}

So is the result for the Heawood family simply a coincidence, or is something deeper involved?

\begin{question}
Under what conditions do $\nabla Y$ and $Y\nabla$ moves preserve intrinsic chirality?
\end{question}

We hope to investigate this question further in future work.

\subsection*{Acknowledgements}

The authors are grateful to Erica Flapan for valuable conversations and insight during the writing of this paper.

\bibliographystyle{amsplain}
\bibliography{TSG.bib}

\end{document}